\documentclass{birkjour}

\usepackage[utf8]{inputenc}
\usepackage[T1]{fontenc}
\usepackage[english]{babel}
\usepackage{mathtools,amssymb,amsthm}
\mathtoolsset{showonlyrefs}
\usepackage{enumitem}

\usepackage{geometry}
\usepackage{color}
\usepackage[pdftex]{hyperref}

\newtheorem{theorem}{Theorem}[section]
\newtheorem{lemma}[theorem]{Lemma}
\newtheorem{proposition}[theorem]{Proposition}

\theoremstyle{definition}
\newtheorem{remark}[theorem]{Remark}

\numberwithin{equation}{section}

\begin{document}

\title[Bifurcation from periodic solutions of central force problems]{Bifurcation from periodic solutions of central \\force problems in the three-dimensional space}

\author[A.~Boscaggin]{Alberto Boscaggin}

\address{
Dipartimento di Matematica ``Giuseppe Peano''\\ 
Universit\`{a} degli Studi di Torino\\
Via Carlo Alberto 10, 10123 Torino, Italy}

\email{alberto.boscaggin@unito.it}

\author[G.~Feltrin]{Guglielmo Feltrin}

\address{
Dipartimento di Scienze Matematiche, Informatiche e Fisiche\\ 
Universit\`{a} degli Studi di Udine\\
Via delle Scienze 206, 33100 Udine, Italy}

\email{guglielmo.feltrin@uniud.it}

\author[D.~Papini]{Duccio Papini}

\address{
Dipartimento di Scienze e Metodi dell'Ingegneria\\
Universit\`{a} degli Studi di Modena e Reggio Emilia\\
Via Giovanni Amendola 2, 42122 Reggio Emilia, Italy}

\email{duccio.papini@unimore.it}

\thanks{Work written under the auspices of the Grup\-po Na\-zio\-na\-le per l'Anali\-si Ma\-te\-ma\-ti\-ca, la Pro\-ba\-bi\-li\-t\`{a} e le lo\-ro Appli\-ca\-zio\-ni (GNAMPA) of the Isti\-tu\-to Na\-zio\-na\-le di Al\-ta Ma\-te\-ma\-ti\-ca (INdAM). 
The first and second authors are supported by the INdAM-GNAMPA Project 2025 ``Existence and symmetry breaking via nonsmooth critical point theory'' and PRIN 2022 ``Pattern formation in nonlinear phenomena''. The first author also acknowledges the support of the PNRR project ``Spoke 3 - Astrophysics and Cosmos Observation''.
\\
\textbf{Preprint -- June 2025}} 



\date{}

\dedicatory{}

\begin{abstract}
The paper deals with electromagnetic perturbations of a central force problem of the form
\begin{equation*}
\dfrac{\mathrm{d}}{\mathrm{d}t} \bigl( \varphi(\dot{x}) \bigr) = V'(|x|) \dfrac{x}{|x|} + E_{\varepsilon}(t,x)+\dot{x} \wedge B_{\varepsilon}(t,x), 
\qquad x \in \mathbb{R}^3 \setminus \{0\}, 
\end{equation*}
where $V \colon (0,+\infty) \to \mathbb{R}$ is a smooth function, $E_\varepsilon$ and $B_\varepsilon$ are respectively the electric field and the magnetic field, smooth and periodic in time, $\varepsilon\in\mathbb{R}$ is a small parameter. 
The considered differential operator includes, as special cases, the classical one, $\varphi(v)=mv$, as well as that of special relativity, $\varphi(v) = mv/\sqrt{1-\vert v \vert^2/c^2}$.
We investigate whether non-circular periodic solutions of the unperturbed problem (i.e., with $\varepsilon=0$) can be continued into periodic solutions for $\varepsilon\neq0$ small, both for the fixed-period problem and, if the perturbation is time-independent, for the fixed-energy problem.
The proof is based on an abstract bifurcation theorem of variational nature, which is applied to suitable Hamiltonian action functionals. In checking the required non-degeneracy conditions we take advantage of the existence of partial action-angle coordinates as provided by the Mishchenko--Fomenko theorem for superintegrable systems. 
Physically relevant problems to which our results can be applied are homogeneous central force problems in classical mechanics and the Kepler problem in special relativity.
\end{abstract}

\maketitle

\section{Introduction}\label{section-1}

In classical mechanics, the motion in the $d$-dimensional Euclidean space ($d \geq 2$) of a particle under the action of a central force is described by the differential equation
\begin{equation}\label{eq-cf}
m \ddot x = V'(|x|) \dfrac{x}{|x|}, \qquad x \in \mathbb{R}^d \setminus \{0\},
\end{equation}
where $m>0$ is the mass of the particle and $V \colon (0,+\infty) \to \mathbb{R}$ is a potential. More recently, central force problems have been also considered in a special relativistic setting; in this case, the equation writes as
\begin{equation}\label{eq-cfrel}
\dfrac{\mathrm{d}}{\mathrm{d}t}\left(\frac{m\dot{x}}{\sqrt{1-|\dot{x}|^{2}/c^{2}}}\right) = V'(|x|) \dfrac{x}{|x|}, \qquad x \in \mathbb{R}^d \setminus \{0\},
\end{equation}
where $c$ is the speed of light, see \cite{BoDaFe-2021,OrRo-25} and the references therein. It is a well known fact that, for both such equations, each orbit lies on a plane passing through the origin (see, for instance, \cite[Appendix~A]{FoTo-08}), so that in the description of the dynamics it is not restrictive to assume $d = 2$. 

Whenever the force is attractive (i.e., $V'<0$), equations \eqref{eq-cf} and \eqref{eq-cfrel} possess circular (that is, $\vert x(t) \vert \equiv R$ for some $R > 0$) periodic solutions and, often, non-circular periodic solutions, as well. Consequently, a natural problem from the point of view of nonlinear analysis and dynamical systems is to investigate whether such periodic solutions can be continued into periodic solutions of perturbed problems of the form  
\begin{equation}\label{eq-cf-pert}
m \ddot x = V'(|x|) \dfrac{x}{|x|} + \nabla_x U_\varepsilon(t,x), \qquad x \in \mathbb{R}^d \setminus \{0\},
\end{equation}
or 
\begin{equation}\label{eq-cfrel-pert}
\dfrac{\mathrm{d}}{\mathrm{d}t}\left(\frac{m\dot{x}}{\sqrt{1-|\dot{x}|^{2}/c^{2}}}\right) = V'(|x|) \dfrac{x}{|x|} + \nabla_x U_\varepsilon(t,x), \qquad x \in \mathbb{R}^d \setminus \{0\},
\end{equation}
where $\varepsilon$ is a small parameter and $U_\varepsilon\colon \mathbb{R} \times (\mathbb{R}^d \setminus \{0\}) \to \mathbb{R}$ is a family of external potentials, periodic in time with the same period of the unperturbed solution, such that $U_0(t,x) \equiv 0$. 
Note that, in spite of the fact that the orbits of the unperturbed problem are planar, in the perturbed setting it is meaningful to take $d \geq 2$ arbitrary, the choice $d=3$ being the most significant from the point of view of the applications.

It is evident that periodic solutions for $\varepsilon = 0$ cannot be isolated: indeed, since \eqref{eq-cf} and \eqref{eq-cfrel} are autonomous
equations, if $x^*(t)$ is a solution, then $x^*(t-\theta)$ is still a solution for every $\theta \in \mathbb{R}$. Accordingly, one typically looks for a so-called \textit{bifurcation}, for $\varepsilon \neq 0$ and small, from a fixed manifold $\mathcal{M}$ of periodic solutions (periodic manifold, for short) of the unperturbed problem. More precisely, we seek a family $\{x_\varepsilon\}$ of periodic solutions of the perturbed problem \eqref{eq-cf-pert} (or \eqref{eq-cfrel-pert}) which stay close, as $\varepsilon \to 0$, to some unperturbed solution $x^* \in \mathcal{M}$. Typically, these bifurcating solutions are required to have either fixed-period (that is, the same period of the unperturbed solution) or, if the external potential $U$ is time-independent, fixed-energy (that is, the same energy of the unperturbed solution, where the notion of energy is defined in agreement with the Lagrangian structure of equations \eqref{eq-cf-pert} and \eqref{eq-cfrel-pert} respectively, see Remark~\ref{rem-contiespliciti} for more details).
As a matter of fact, the crucial condition ensuring that a bifurcation actually exists is a non-degeneracy assumption for the periodic manifold $\mathcal{M}$.

The first results in this spirit seem to be the one contained in the pioneering paper \cite{AmCo-89}. Therein, an abstract bifurcation theorem of variational nature, previously established in \cite{AmCoEk-87}, was used to prove the existence of periodic solutions, bifurcating from circular solutions, for equations of the type \eqref{eq-cf-pert}, in the fixed-period case. More precisely, Keplerian-type potentials $V(r) = \kappa/(\alpha r^\alpha)$, with $\kappa,\alpha > 0$, were considered in arbitrary dimension $d \geq 2$. The validity of the required non-degeneracy condition, involving the linearization of equation \eqref{eq-cf} along a circular solution, was investigated using a Fourier series argument.
It is worth mentioning that the Keplerian case $\alpha = 1$ was left out of this analysis, since the superintegrable character of the Kepler problem determines the failure of the non-degeneracy condition. Similar results in the fixed-energy case were later obtained in \cite{AmBe-92}, using the same abstract theorem as in \cite{AmCo-89}, applied however to the Maupertuis functional instead of the Lagrangian action one. See also \cite{BA-23,FoGa-18} and the references therein for other contributions on this line of research. 
Bifurcation from circular solutions for the relativistic problem \eqref{eq-cfrel-pert}, instead, has been considered in \cite{BoFePa-25} (see also \cite{Ga-19} for a previous related result) for the Kepler potential $V(r) = \kappa/r$ in dimension $d=2$ and $d=3$, both in the case of fixed-period (again via the same abstract theorem as in \cite{AmCo-89}) and in that of fixed-energy (via a different bifurcation theorem from periodic manifolds established in \cite{We-73,We-78}). Incidentally, note that these results show that the degeneracy of the Kepler problem is broken when introducing a relativistic correction. 

The problem of bifurcation from \textit{non-circular} periodic solutions of central force problems, on the other hand, has a much more recent history and has been considered, in dimension $d=2$, starting with the paper  \cite{BoDaFe-2021}.
Therein, for the relativistic equation \eqref{eq-cfrel-pert} with Kepler potential $V(r) = \kappa/r$, the occurrence of bifurcation was proved from the manifold
\begin{equation}\label{defM2-intro}
\mathcal{M}_2 = \bigl{\{} e^{i\phi} x^*(t-\theta) \colon \phi,\theta \in \mathbb{R} \bigr{\}},
\end{equation}
made up by time-translations and planar rotations of a fixed non-circular (and non-rectilinear) periodic solution $x^*$ of the unperturbed problem.
The periodic manifold $\mathcal{M}_2$ is homeomorphic to a two-dimensional torus and in fact it corresponds to (a connected component of) the level set of the energy and the (relativistic) angular momentum: these are the two natural first integrals of equation \eqref{eq-cfrel} in the planar case $d=2$. Accordingly, in a sort of periodic counterpart of KAM theory, bifurcation from the torus $\mathcal{M}_2$ was obtained, using a higher-dimensional version of the Poincar\'e--Birkhoff theorem \cite{FoUr-17} (see also \cite{FoGaGi-16}), after proving that the KAM non-degeneracy condition
\begin{equation}\label{KAMintro}
\operatorname{det} \nabla^2 \mathcal{K}_0(I^*) \neq 0
\end{equation}
is satisfied, where $\mathcal{K}_0$ is the Hamiltonian of the unperturbed problem in action-angle coordinates
$(I,\varphi) \in \mathbb{R}^2 \times \mathbb{T}^2$, and $I^*$ is the action-value corresponding to the fixed torus $\mathcal{M}_2$.
We emphasize that this approach was greatly facilitated by the rather unusual fact that the expression of the Hamiltonian $\mathcal{K}_0$
is fully explicit for the relativistic Kepler problem. Actually, in the subsequent paper \cite{BDF-2}, this general strategy was refined and, taking advantage of an equivalent reformulation of the non-degeneracy condition \eqref{KAMintro} proved therein, bifurcation from the periodic manifold $\mathcal{M}_{2}$ was established for the perturbed problem \eqref{eq-cf-pert} in the case of homogeneous potentials $V(r) = \kappa/(\alpha r^\alpha)$ for every $\alpha <2$, with $\alpha \notin \{-2,0,1\}$.
In \cite{BDF-3} (see also \cite{BoDaFe-24ch}), counterparts of these results were given for the fixed-energy problem, using an energy reduction procedure together with the classical planar version of the Poincar\'e--Birkhoff theorem. In this case, the non-degeneracy condition to be checked turns out to be
\begin{equation}\label{KAMintro2}
\operatorname{det}
\begin{pmatrix}
\nabla^2 \mathcal{K}_{0}(I^{*}) & \nabla \mathcal{K}_{0}(I^{*})^{\top} \vspace{3pt} \\
\nabla \mathcal{K}_{0}(I^{*}) & 0
\end{pmatrix} 
\neq 0,
\end{equation}
corresponding to the so-called isoenergetic KAM non-degeneracy condition (see \cite[Appendix~8]{Ar-89} and \cite{BrHu-91}). We stress that all these contributions deal with the planar case $d=2$.

Motivated by the above described achievements, in this paper we investigate bifurcation from non-circular periodic solutions of central force problems in the \textit{three-dimensional space}. More precisely, we deal with the differential equation
\begin{equation}\label{eq-main}
\dfrac{\mathrm{d}}{\mathrm{d}t} \bigl( \varphi(\dot{x}) \bigr) = V'(|x|) \dfrac{x}{|x|} + E_{\varepsilon}(t,x)+\dot{x} \wedge B_{\varepsilon}(t,x), \qquad x \in \mathbb{R}^3 \setminus \{0\}, 
\end{equation}
where, here and throughout the paper, we assume that:
\begin{enumerate}[leftmargin=32pt,labelsep=10pt,label=\textup{$(H_{1})$}]
\item $\varphi \colon B_{a}(0) \to B_{b}(0)$ (with $a,b \in (0,+\infty]$) is a homeomorphism such that $\varphi(0)=0$ and 
\begin{equation*}
\varphi(v) = f(|v|) \dfrac{v}{|v|}, \qquad v\neq0,
\end{equation*}
where $f \colon [0,a) \to [0,b)$ is a continuous function, continuously differentiable on $(0,a)$, and such that
\label{H1}
\begin{equation*}
f(0)=0, 
\qquad 
f'(s)>0, \; \text{for every $s\in(0,a)$,} 
\qquad 
\lim_{s\to a^{-}} f(s) = b;
\end{equation*}
\end{enumerate}
\begin{enumerate}[leftmargin=32pt,labelsep=10pt,label=\textup{$(H_{2})$}]
\item $V \colon (0,+\infty) \to \mathbb{R}$ is of class $\mathcal{C}^2$;
\label{H2}
\end{enumerate}
\begin{enumerate}[leftmargin=32pt,labelsep=10pt,label=\textup{$(H_{3})$}]
\item $E_\varepsilon \colon \mathbb{R} \times (\mathbb{R}^3 \setminus \{0\}) \to \mathbb{R}^3$ and $B_\varepsilon \colon \mathbb{R} \times (\mathbb{R}^3 \setminus \{0\}) \to \mathbb{R}^3$ are given, for small $\varepsilon$, by
\begin{equation}\label{eq-maxwell}
E_{\varepsilon}(t,x)=\nabla_{x}U_{\varepsilon}(t,x) - \dfrac{\partial}{\partial t} A_{\varepsilon}(t,x),
\qquad
B_{\varepsilon}(t,x)=\mathrm{curl}_{x} \, A_{\varepsilon}(t,x),
\end{equation}
where $U_\varepsilon\colon \mathbb{R} \times (\mathbb{R}^3 \setminus \{0\}) \to \mathbb{R}$ 
and $A_\varepsilon\colon \mathbb{R} \times (\mathbb{R}^3 \setminus \{0\}) \to \mathbb{R}^3$ are $T$-periodic in $t$, two times differentiable with respect to $x$, and $A_\varepsilon$ is differentiable with respect to $t$; moreover, $U_\varepsilon$ and $A_\varepsilon$ and their derivatives are continuous in $(t,x)$ and depend smoothly on $\varepsilon$, with $U_0(t,x) \equiv 0$ and $A_0(t,x) \equiv 0$.
\label{H3}
\end{enumerate}
Note that the above setting is general enough to include perturbations of central force problems in classical mechanics as well as in the special relativistic scenario, by taking, in assumption \ref{H1}, $f(s) = ms$ and $f(s) = ms/\sqrt{1-s^2/c^2}$, respectively. We also observe that the structure for the perturbation term described in \ref{H3}, which finds its motivation in the theory of electromagnetism (indeed, $E_\varepsilon$ and $B_\varepsilon$ can be interpreted, respectively, as an electric field and a magnetic field, given, via \eqref{eq-maxwell}, by the electrostatic potential $U_\varepsilon$ and the magnetic vector potential $A_\varepsilon$) is slightly more general with respect to previously mentioned works, where $A_\varepsilon \equiv 0$ (and, thus, $E_\varepsilon= \nabla_x U_\varepsilon$ and $B_\varepsilon \equiv 0$). 
Equation \eqref{eq-main} is still of Lagrangian type (see Section~\ref{section-2.1} for the details), with associated Lagrangian function
\begin{equation}\label{eq-deflag}
\mathcal{L}_\varepsilon(t,x,v) = F(\vert v \vert) + \langle v , A_\varepsilon(t,x) \rangle + V(\vert x \vert) + U_\varepsilon(t,x),
\end{equation}
where $v=\dot{x}$ and
\begin{equation*}
F(s) = \int_{0}^{s} f(\xi)\,\mathrm{d}\xi, \quad s\in [0,a).
\end{equation*}
Accordingly, the energy can be defined as
\begin{equation}\label{eq-defenergy}
\mathcal{E}_\varepsilon(t,x,v) = G(\vert \varphi(v) \vert) - V(\vert x \vert) - U_\varepsilon(t,x),
\end{equation}
where
\begin{equation}\label{eq-defG}
G(s) = \int_{0}^{s} f^{-1}(\xi)\,\mathrm{d}\xi, \quad s\in [0,b),
\end{equation}
with $f^{-1}$ the inverse of $f$.
As well known, if equation \eqref{eq-main} is autonomous, then
$\mathcal{E}_\varepsilon(t,x,v) = \mathcal{E}_\varepsilon(x,v)$ is a first integral, that is,  
$\mathcal{E}_\varepsilon(x(t),\dot x(t))$ is constant along a solution $x$ of  \eqref{eq-main}.

In this setting, we study, both in the fixed-period case and in the fixed-energy one, bifurcation from the manifold
\begin{equation}\label{defM3}
\mathcal{M}_3 =
\bigl{\{} M x^*(t-\theta) \colon M\in O(3), \theta \in \mathbb{R} \bigr{\}},
\end{equation}
made up by time-translations and orthogonal transformations of a fixed non-circular (and non-rectilinear) periodic solution $x^*$ of the unperturbed problem
\begin{equation}\label{eq-intro-unper}
\dfrac{\mathrm{d}}{\mathrm{d}t} \bigl( \varphi(\dot{x}) \bigr) = V'(|x|) \dfrac{x}{|x|}, 
\qquad x \in \mathbb{R}^d \setminus \{0\},
\end{equation}
with $d=3$.
It turns out (see Proposition~\ref{prop-2.1}) that $\mathcal{M}_3$ is a compact manifold of dimension four, homeomorphic to $SO(3) \times \mathbb{T}^1$. We also remark that, since the orbit of $x^*$ is planar, it is not restrictive to assume that $x^*$ actually belongs to the horizontal plane $x_3 = 0$.
With this in mind, we are going to prove (see Theorem~\ref{th-abstract1} and Theorem~\ref{th-abstract2} for the precise statement) that:
\begin{quote}
\textit{If the unperturbed problem \eqref{eq-intro-unper} is non-degenerate as a planar problem, then bifurcation from $\mathcal{M}_3$ occurs for the spatial problem \eqref{eq-main}.}
\end{quote}
By ``non-degenerate as a planar problem'', we mean that the periodic manifold $\mathcal{M}_2$ defined in \eqref{defM2-intro}, made up by solutions of the unperturbed problem \eqref{eq-intro-unper} lying on the horizontal plane (that is, equivalently, solutions of \eqref{eq-intro-unper} for $d=2$), is non-degenerate in the sense considered in \cite{BDF-2} and \cite{BDF-3}, that is, conditions \eqref{KAMintro} and \eqref{KAMintro2} hold, for the fixed-period and fixed-energy problem respectively. Incidentally, we mention that, as observed in Section~\ref{section-3.2}, these non-degeneracy conditions in action-angle coordinates can be equivalently expressed in terms of the dimension of the space of $T$-periodic solutions of a suitable linearized equation along $x^*$ (in particular, dealing with  the easiest case of the fixed-period problem, condition \eqref{KAMintro} is satisfied if and only if the linearization of \eqref{eq-intro-unper} along the planar solution $x^*$ has a $2$-dimensional space of horizontal $T$-periodic solutions). We also observe that, since the Lusternik--Schnirelmann category of $\mathcal{M}_3$ can be exactly computed (being equal to $5$, see again Proposition~\ref{prop-2.1}), a sharp lower bound for the number of bifurcating solutions can be given in the case of the fixed-period problem, see also Remark~\ref{rem-cat-multiplicity}.

Having reduced the problem of bifurcation from $\mathcal{M}_3$ to a non-degeneracy condition for a planar problem, our results find immediate applications to any three-dimensional equation of the type \eqref{eq-main} for which the non-degeneracy of the unperturbed problem has been established in the planar setting. In particular, in view of the results in \cite{BDF-2} (fixed-period) and \cite{BoDaFe-24ch,BDF-3} (fixed-energy),
in the context of classical mechanics we can deal with spatial perturbations of the Levi-Civita equation 
\begin{equation*}
m \ddot x = -\kappa \dfrac{x}{|x|^{3}} - 2\lambda \dfrac{x}{|x|^4}, \qquad \kappa,\lambda > 0,
\end{equation*}
introduced in \cite{LC-28} as a relativistic correction of the Kepler problem, as well as spatial perturbations of the homogeneous central force problem 
\begin{equation*}
m \ddot x = -\kappa \dfrac{x}{|x|^{\alpha + 2}}, \qquad \kappa > 0,\, \alpha < 2, \,\alpha \notin \{-2,1\}.
\end{equation*}
Incidentally, we point out that this result is essentially optimal, since for $\alpha \geq 2$ the unperturbed problem does not have non-circular periodic solutions \cite[Section~2.b]{AmCo-93}, while the other excluded values of $\alpha$ correspond to the well known degenerate cases of the harmonic oscillator 
($\alpha = -2$) and of the Kepler problem ($\alpha =  1$). As for relativistic mechanics, instead, we can deal with spatial perturbations of the relativistic Kepler problem
\begin{equation*}
\dfrac{\mathrm{d}}{\mathrm{d}t}\left(\frac{m\dot{x}}{\sqrt{1-|\dot{x}|^{2}/c^{2}}}\right) = -\kappa \dfrac{x}{|x|^{3}}, 
\qquad \kappa > 0,
\end{equation*}
whose non-degeneracy in the plane has been proved in \cite{BoDaFe-2021} (fixed-period) and 
\cite{BoDaFe-24ch} (fixed-energy). 
Incidentally, we observe that exploring the non-degeneracy of the relativistic problem with different homogeneous potentials (that is $V(r) = \kappa / (\alpha r^\alpha)$ with $\alpha\neq1$) seems to be a rather delicate issue which could deserve further investigation.

The proof of our main results is performed using the abstract theorem in \cite{AmCoEk-87} (see Theorem~\ref{thastratto1}), providing bifurcation (from a fixed critical manifold) of critical points for a family of $\mathcal{C}^2$-functionals  $\mathcal{F}_\varepsilon$ defined on an open subset of a Hilbert space $H$. It is worth noticing that here, due to the lack of the $\mathcal{C}^2$-smoothness of the Lagrangian action associated with \eqref{eq-main} when $\varphi(v) \neq mv$ (see \cite[Remark~3.3]{BoFePa-25} and \cite{AbSc-01}), we are forced to first pass to the Hamiltonian formulation and then consider the associated Hamiltonian action functional. This strategy has already been used in the papers \cite{BoFePa-25,Ga-19} dealing with the fixed-period problem, and thus considering a functional of the type
\begin{equation*}
 \mathcal{A}_{\varepsilon}(z) = \frac{1}{2}  \int_0^T \langle J z'(t), w(t) \rangle \,\mathrm{d}t- \int_0^T \mathcal{H}_{\varepsilon}(t,z(t))\,\mathrm{d}t,
\end{equation*}
on the fractional Sobolev space $H = H^{1/2}_T$, see Section~\ref{section-3.1} for the details. 
It is worth noticing that, due to the facts that the Hamiltonian $\mathcal{H}_{\varepsilon}$ is not globally defined and that the functions in $H^{1/2}$ are not bounded in general, a suitable cut-off procedure has to be implemented in order to recover $L^\infty$-bounds.
In this paper, similar ideas are used for the fixed-energy problem: in particular 
(see again Section~\ref{section-3.1}) the free-period Hamiltonian action functional
\begin{equation*}
 \mathcal{B}_{\varepsilon}(\zeta,T) =  \frac{1}{2}  \int_0^1 \langle J \zeta'(s), \zeta(s) \rangle \,\mathrm{d}s - T\int_0^1 \bigl( \mathcal{H}_{\varepsilon}(\zeta(s)) - h \bigr)\,\mathrm{d}s
\end{equation*}
 must be considered. We stress that, even if the above functional (also known as Rabinowitz action functional, see \cite{Ra-79}) is very popular in the context of Floer homology \cite{AlFr-12}, its use in the more elementary framework of the fractional Sobolev space $H^{1/2}$ seems to be much less explored. 
 
Within this variational setting, the core of the proof consists in showing that the non-degeneracy condition for the periodic manifold $\mathcal{M}_2$ (which is assumed as a hypothesis) implies that $\mathcal{M}_3$ satisfies the non-degeneracy condition required for the application of the abstract bifurcation result.
To achieve this goal, we exploit the fact that the unperturbed spatial problem is a superintegrable Hamiltonian system with $3$ degrees of freedom, having as independent first integrals the energy and the three components of the vector angular momentum (see Section~\ref{section-2.3} for more details). Accordingly, we make use of a suitable symplectic change of coordinates transforming (locally) the unperturbed problem into a system of the form
\begin{equation*}
\begin{cases}
\, \dot I_i = 0, \qquad \dot \phi_i = \partial_{I_i} \mathcal{K}_0(I_1,I_2), \qquad i=1,2, \vspace{2pt}\\
\, \dot \Xi = 0, \qquad \dot \xi = 0,
\end{cases}
\end{equation*}
where $I_1,I_2,\Xi,\xi$ are real variables and $\phi_1,\phi_2\in\mathbb{T}^{1}$. The above system of coordinates is provided by the Mishchenko--Fomenko theorem for superintegrable systems \cite{MiFo-78} (see also \cite{Fa-05}). Roughly speaking, it can be seen as an extension of the traditional system of action-angle coordinates given by the classical Liouville--Arnold theorem for integrable systems, the novelty being in the presence of the coordinates $(\Xi,\xi)$, which do not form an action-angle pair.
On the other hand, $(I_1,\phi_1,I_2,\phi_2)$ are nothing but the usual action-angle coordinates of the planar problem, with $\mathcal{K}_0$ the associated Hamiltonian. In this way, the non-degeneracy of $\mathcal{M}_3$ is reduced to a condition on the planar Hamiltonian $\mathcal{K}_0$.

We end this introduction by observing the following quite subtle fact. While for the unperturbed planar problem the non-degeneracy condition required for bifurcation from $\mathcal{M}_{2}$ coincides with the non-degeneracy condition of KAM theory (implying the persistence of tori made by quasi-periodic solutions), the unperturbed problem in the $3$-dimensional space is always KAM-degenerate (since it is superintegrable). In spite of this, the non-degeneracy condition required for bifurcation from $\mathcal{M}_{3}$ can be satisfied in some cases, precisely the ones in which $\mathcal{M}_{2}$ is non-degenerate.

\medskip

The plan of the paper is the following. In Section~\ref{section-2.1}, we discuss the Hamiltonian structure of the perturbed equation \eqref{eq-main}. Next, in Section~\ref{section-2+1} we focus our attention on the unperturbed equation \eqref{eq-intro-unper}, in the plane and in the three-dimensional space; in both cases, we illustrate the structure of the manifold of non-circular periodic solutions and the associated action-angles coordinates.
In Section~\ref{section-3}, we deal with general Hamiltonian systems in $\mathbb{R}^{2N}$, presenting a variational formulation and an overview of the notion of non-degeneracy, both for the fixed-period and fixed-energy problems.
Finally, the precise statements and proofs of our main bifurcation results are given in Section~\ref{section-4}.

\medskip

Throughout the paper, by $\vert \cdot \vert$ we denote the Euclidean norm of a vector in $\mathbb{R}^d$, with $d \geq 2$,
and by $\langle \cdot, \cdot \rangle$ the associated scalar product. Dealing with differential equations like the ones above considered, by \textit{circular} solution we always mean a solution $x$ such that
$|x(t)|=R$ for every $t$ (for some $R>0$), while a solution is called \textit{rectilinear} if $x(t)/|x(t)| = c$ for every $t$ (for some versor $c$).
By the \textit{Lusternik--Schnirelmann category} of a topological space $X$, denoted by $\operatorname{cat}(X)$, we mean the least number of open contractible sets needed to cover $X$. 
If $X$ is a binormal ANR (in particular if $X$ is a topological manifold) the above definition coincides with the one usually considered in nonlinear analysis, namely the least number of closed contractible sets needed to cover $X$ (see for instance \cite[Proposition~4.3]{RuSc-03}).

\section{Hamiltonian formulation}\label{section-2.1}

Let us consider equation \eqref{eq-main}. Recalling the vector calculus formula
\begin{equation}\label{eq-vectorcalculus}
(\mathrm{D}_{x}A_{\varepsilon}(t,x))^{\top} y - \mathrm{D}_{x}A_{\varepsilon}(t,x) y = y \wedge \mathrm{curl}_{x} \, A_{\varepsilon}(t,x),
\end{equation}
valid for every $t \in [0,T]$ and $x,y \in \mathbb{R}^3$, it is easily checked that equation \eqref{eq-main} can be written as the following first-order system in $\mathbb{R}^6$
\begin{equation}\label{eq-HS}
\begin{cases}
\, \dot{x} = \varphi^{-1}(p-A_{\varepsilon}(t,x)),
\vspace{4pt}\\
\, \dot{p} = V'(|x|) \dfrac{x}{|x|}+\nabla_{x}U_{\varepsilon}(t,x) + (\mathrm{D}_{x}A_{\varepsilon}(t,x))^{\top}\varphi^{-1}(p-A_{\varepsilon}(t,x)).
\end{cases}
\end{equation}
Notice that $\varphi^{-1} \colon B_{b}(0) \to B_{a}(0)$ is given by
\begin{equation*}
\varphi^{-1}(w) = f^{-1}(|w|) \dfrac{w}{|w|} = \nabla G(|w|), \quad w\neq0,
\end{equation*}
where $G$ is as in \eqref{eq-defG}. From this we immediately deduce that \eqref{eq-HS} is of Hamiltonian type 
\begin{equation}\label{eq-system}
\dot{x} = \nabla_{p} \mathcal{H}_{\varepsilon}(t,x,p),
\qquad
\dot{p} = -\nabla_{x} \mathcal{H}_{\varepsilon}(t,x,p),
\end{equation}
with Hamiltonian function
\begin{equation}\label{def-Ham}
\mathcal{H}_{\varepsilon}(t,x,p) = G(|p-A_{\varepsilon}(t,x)|)-V(|x|)-U_{\varepsilon}(t,x).
\end{equation}
Note that $\mathcal{H}_\varepsilon$ is $T$-periodic in time, and two times continuously differentiable with respect to $(x,p)$ as long as $x \neq 0$ and $p \neq A_{\varepsilon}(t,x)$. The energy $\mathcal{E}_\varepsilon$ defined in \eqref{eq-defenergy} is nothing but the Hamiltonian $\mathcal{H}_\varepsilon$  written in terms of the velocity 
\begin{equation*}
v = \dot x = \varphi^{-1}(p-A_{\varepsilon}(t,x))
\end{equation*}
instead of the momentum $p$, that is
\begin{equation*}
\mathcal{E}_\varepsilon(t,x,v) = \mathcal{H}_{\varepsilon}(t,x,\varphi(v) + A_{\varepsilon}(t,x) ) = G(|\varphi(v)|)-V(|x|)-U_{\varepsilon}(t,x).
\end{equation*}
Incidentally, notice that the function $A_{\varepsilon}$ does not appear explicitly in the definition of $\mathcal{E}_\varepsilon$.
When both $A_{\varepsilon}$ and $U_{\varepsilon}$ do not depend on time (that is, system \eqref{eq-system} is autonomous), the Hamiltonian $\mathcal{H}_{\varepsilon}(t,x,p) = \mathcal{H}_{\varepsilon}(x,p)$ (and, thus, the energy 
$\mathcal{E}_\varepsilon(t,x,v) = \mathcal{E}_\varepsilon(x,v)$) are constant along the solutions of \eqref{eq-system} (note the rather subtle fact that, if $A_\varepsilon$ is time-dependent while $U_\varepsilon$ is not, the energy $\mathcal{E}_\varepsilon(t,x,v)$ does not depend explicitly on the time-variable $t$, but is \textit{not} a first integral).

When $\varepsilon = 0$, one finds the unperturbed (autonomous) 
Hamiltonian
\begin{equation}\label{eq-ham0}
\mathcal{H}_{0}(t,x,p) = \mathcal{H}_0(x,p) = G(|p|)-V(|x|)
\end{equation}
and the unperturbed system
\begin{equation}\label{eq-HS0}
\begin{cases}
\, \dot{x} = \varphi^{-1}(p),
\vspace{4pt}\\
\, \dot{p} = V'(|x|) \dfrac{x}{|x|},
\end{cases}
\qquad (x,p) \in (\mathbb{R}^{d} \setminus \{0\} ) \times \mathbb{R}^d,
\end{equation}
where $d=3$. However, in what follows, we will also need to consider system \eqref{eq-HS0} in the planar case $d=2$.

\begin{remark}
Using formula \eqref{eq-vectorcalculus}, it is immediate to verify that \eqref{eq-main} is the Lagrange equation of the 
Lagrangian $\mathcal{L}_\varepsilon(t,x,\dot x)$ defined in \eqref{eq-deflag}. The Hamiltonian $\mathcal{H}_\varepsilon(t,x,p)$
is, of course, nothing but the Legendre transform (with respect to $v$) of the Lagrangian $\mathcal{L}_\varepsilon(t,x,v)$, that is
\begin{equation*}
\mathcal{H}_\varepsilon(t,x,p) = \bigl(\langle v, p\rangle - \mathcal{L}_\varepsilon(t,x,v)\bigr)|_{v = \varphi^{-1}(p-A_\varepsilon(t,x))}.
\end{equation*}
Indeed, by simple computations the above formula gives
\begin{align*}
\mathcal{H}_\varepsilon(t,x,p) & = \langle \varphi^{-1}(p-A_\varepsilon(t,x)), p-A_\varepsilon(t,x)\rangle - 
F(\vert \varphi^{-1}(p-A_\varepsilon(t,x))\vert) \\
& \quad - V(|x|)-U_{\varepsilon}(t,x) \\
& = f^{-1}(\vert p-A_\varepsilon(t,x) \vert) \vert p-A_\varepsilon(t,x) \vert - F(f^{-1}(\vert p-A_\varepsilon(t,x) \vert)) \\
& \quad - V(|x|)-U_{\varepsilon}(t,x) ,
\end{align*}
and, since
\begin{equation*}
G(s) = f^{-1}(s)s - F(f^{-1}(s)),
\end{equation*}
we immediately recover the expression of $\mathcal{H}_\varepsilon(t,x,p)$ given in \eqref{def-Ham}.
\hfill$\lhd$
\end{remark}

\begin{remark}\label{rem-contiespliciti}
For the reader's convenience, and in order to facilitate possible comparisons with existing results in the literature, we write down below the explicit expressions of the unperturbed Lagrangian, energy and Hamiltonian in the case of classical mechanics ($f(s) = ms$) and relativistic mechanics ($f(s) = ms/\sqrt{1-s^2/c^2}$). Precisely, in the classical case one finds
\begin{equation*}
F(s) = \frac{m}{2}s^2, \quad f^{-1}(s) = \frac{s}{m}, \quad G(s) = \frac{s^{2}}{2m},
\end{equation*}
so that 
\begin{align*}
\mathcal{L}_0(x, v ) & = \frac{m}{2} \vert v \vert^2 + V(\vert x \vert), \\
\mathcal{E}_0(x, v) & =  \frac{m}{2} \vert v \vert^2 - V(\vert x \vert), \\
\mathcal{H}_0(x, p) & = \frac{\vert p \vert^2}{2m} -  V(\vert x \vert).
\end{align*}
On the other hand, in the case of special relativity,
\begin{equation*}
F(s) =mc^2 \left( 1 - \sqrt{1-\frac{s^2}{c^2}}\right), \quad f^{-1}(s) = \frac{s}{m \sqrt{1+s^2/(m^2c^2)}}, 
\end{equation*}
and
\begin{equation*} 
G(s) =mc^2 \left( \sqrt{1+ \frac{s^2}{m^2c^2}}-1\right),
\end{equation*}
so that
\begin{align*}
\mathcal{L}_0(x, v ) & = mc^2 \left( 1 - \sqrt{1-\frac{\vert v\vert^2}{c^2}}\right) + V(\vert x \vert), \\
\mathcal{E}_0(x, v) & =  mc^2\left( \frac{1}{\sqrt{1-\vert v\vert^2/c^2}} -1 \right) - V(\vert x \vert), \\
\mathcal{H}_0(x, p) & = mc^2 \left(\sqrt{1 + \frac{\vert p \vert^2}{m^2 c^2}} - 1 \right) -  V(\vert x \vert).
\end{align*}
Note that, in the non-relativistic limit $c \to +\infty$, all the above functions reduce to the classical ones.
\hfill$\lhd$
\end{remark}

\section{Non-circular periodic solutions and action-angles coordinates}\label{section-2+1}

In this section, we discuss the structure of the manifold of non-circular periodic solutions for the unperturbed system \eqref{eq-HS0} as well as the canonical transformation in action-angle coordinates.
More precisely, after having recalled some well-established facts dealing with the $2$-dimensional case in Section~\ref{section-2.2}, in Section~\ref{section-2.3} we analyze in detail the $3$-dimensional case, providing the crucial tools needed in the proof of our main theorems in Section~\ref{section-4}.

\subsection{The $2$-dimensional case}\label{section-2.2}

Let us focus on the unperturbed system \eqref{eq-HS0} in the case $d=2$, that is, equivalently, the planar equation
\begin{equation}\label{eq-02}
\dfrac{\mathrm{d}}{\mathrm{d}t} \bigl( \varphi(\dot{x}) \bigr) = V'(|x|) \dfrac{x}{|x|}, \qquad x \in \mathbb{R}^2 \setminus\{0\}.
\end{equation}
As well known, in this situation there are two first integrals: the Hamiltonian $\mathcal{H}_0$, cf.~\eqref{eq-ham0}, and the scalar angular momentum
\begin{equation}\label{eq-momentum}
\mathcal{L}(x,p) = \langle Jx, p \rangle = x_1 p_2 - x_2 p_1,
\end{equation}
with $x = (x_1,x_2)$, $p=(p_1,p_2)$ and $J$ the standard symplectic matrix 
\begin{equation}\label{eq-standsym}
J = \begin{pmatrix}
0 & -1 \\
1 & 0
\end{pmatrix}.
\end{equation}
Let us assume now that there is a non-circular non-rectilinear periodic solution $x^*$ of equation \eqref{eq-02}, and let
$(H^*,L^*) \in \mathbb{R}^2$ be the associated values of energy and angular momentum. 
Then, the set of functions made up by time-translations and planar rotations of $x^*$, namely
\begin{equation}\label{def-M2}
\mathcal{M}_2 = \bigl{\{} e^{i\phi} x^*(t-\theta) \colon \phi,\theta \in \mathbb{R} \bigr{\}},
\end{equation}
is a manifold of non-circular non-rectilinear periodic solutions to \eqref{eq-02}, which is homeomorphic to the two-dimensional torus $\mathbb{T}^2$.
We observe that in order to have a one-to-one parametrization of $\mathcal{M}_2$ it is sufficient to take $\phi\in[0,2\pi)$ and $\theta\in[0,\tau^*)$, where $\tau^*$ is the minimal period of $|x^*|$.

Let us consider the map
\begin{equation}\label{eq-projection}
\mathcal{M}_2 \ni x(t) \mapsto \Pi(x(t)) = (x(0),p(0)) = (x(0),\varphi(\dot x(0))) \in \mathbb{R}^4.
\end{equation}
It is easy to check that
\begin{equation*}
\Pi(\mathcal{M}_2) \subset \mathcal{T}_{(H^*,L^*)} 
\coloneqq 
\bigl{\{} (x,p) \in \mathbb{R}^4 \colon \mathcal{H}_0(x,p) = H^*, \, \mathcal{L}(x,p) = L^* \bigr{\}}.
\end{equation*}
Actually, the map $\Pi$ provides a homeomorphism of $\mathcal{M}_2$ onto the connected component $\mathcal{T}^*_{(H^*,L^*)}$ of the level set 
$\mathcal{T}_{(H^*,L^*)}$ which contains the value $(x^*(0),p^*(0))$, where $p^*(0) = \varphi(\dot x^*(0))$. 
Since on $\mathcal{T}^*_{(H^*,L^*)}$ the first integrals $\mathcal{H}_0$ and $\mathcal{L}$
are independent and in involution, the Liouville--Arnold theorem (see, for instance, \cite[Theorem~3.4]{MoZe-05}) implies that
there exist an open set $\mathcal{D}_{\mathbb{R}^2} \subset \mathbb{R}^2$, an open neighborhood $\mathcal{D}'_{\mathbb{R}^2} \subset \mathbb{R}^2$ of $(H^*,L^*)$, an open neighborhood $\mathcal{U}_{\mathbb{R}^4} \subset \mathbb{R}^4$ of $\mathcal{T}^*_{(H^*,L^*)}$, a diffeomorphism
\begin{equation*}
\eta \colon \mathcal{D}_{\mathbb{R}^2} \to \mathcal{D}'_{\mathbb{R}^2}, \qquad (I_1,I_2) \mapsto (\eta_1(I_1,I_2),\eta_2(I_1,I_2)), 
\end{equation*}
and a symplectic diffeomorphism
\begin{equation*}
\Psi \colon \mathcal{D}_{\mathbb{R}^2} \times \mathbb{T}^2 \to \mathcal{U}_{\mathbb{R}^4}, \qquad 
 (I_1,I_2,\phi_1,\phi_2) \mapsto (x,p),
\end{equation*}
satisfying, for every $I = (I_1,I_2) \in \mathcal{D}_{\mathbb{R}^2}$,
\begin{equation*}
\Psi \bigl( \{ I \} \times \mathbb{T}^2 \bigr) =  \bigl{\{} (x,p) \in \mathbb{R}^4 \colon \mathcal{H}_0(x,p) = \eta_1(I), \,\mathcal{L}(x,p) = \eta_2(I)\bigr{\}} \cap \mathcal{U}_{\mathbb{R}^4},
\end{equation*}
and transforming, on $\mathcal{U}_{\mathbb{R}^4}$, system \eqref{eq-02} into the canonical form
\begin{equation}\label{action-angle}
\dot I_i = 0, \qquad \dot \phi_i = \partial_{I_i} \mathcal{K}_0(I_1,I_2), \qquad i=1,2,
\end{equation}
for a suitable function $\mathcal{K}_0 \colon \mathcal{D}_{\mathbb{R}^2} \to \mathbb{R}$. 
Notice that, setting $I^* = \eta^{-1}(H^*,L^*)$, it holds that
\begin{equation}\label{def-Istar}
\Psi\bigl(  \{ I^* \} \times \mathbb{T}^2 \bigr)  = \mathcal{T}^*_{(H^*,L^*)} \cong \mathcal{M}_2.
\end{equation}
The coordinates $(I_1,I_2,\phi_1,\phi_2)$
are the so-called action-angle coordinates. Their construction (in our specific case of system \eqref{eq-HS0} with $d=2$) can be made in a quite explicit way: more details on this are however not needed in the rest of the paper, and we refer to \cite{BeFa-notes,BDF-2} for a thorough discussion.

\subsection{The $3$-dimensional case}\label{section-2.3}

We now turn our attention on system \eqref{eq-HS0} in the case $d=3$, that is, equivalently, the spatial equation
\begin{equation}\label{eq-03}
\dfrac{\mathrm{d}}{\mathrm{d}t} \bigl( \varphi(\dot{x}) \bigr) = V'(|x|) \dfrac{x}{|x|}, \qquad x \in \mathbb{R}^3\setminus\{0\}.
\end{equation}
In this case, there are four independent first integrals: the Hamiltonian $\mathcal{H}_0$ and the three components 
of the (vector) angular momentum
\begin{equation*}
\vec{\mathcal{L}}(x,p) = x \wedge p.
\end{equation*}
Notice that, denoting by $\mathcal{L}_i$ ($i=1,2,3$) the three components of $\vec{\mathcal{L}}$, it holds that $\mathcal{L}_3(x,p) = \mathcal{L}(\pi(x),\pi(p))$, where $\mathcal{L}$ is defined in \eqref{eq-momentum} and 
\begin{equation*}
\pi \colon \mathbb{R}^3 \to \mathbb{R}^2, \qquad (y_1,y_2,y_3) \mapsto (y_1,y_2).
\end{equation*}

Having more first integrals than degrees of freedom, system \eqref{eq-HS0} in the case $d=3$ is a so-called superintegrable system. Similarly to the standard integrable case, also superintegrable systems can be written in a canonical form, which is now in terms of so-called generalized (or partial) action-angle coordinates: the associated theory basically goes back to the Mishchenko--Fomenko theorem \cite{MiFo-78} and has been accurately revisited in \cite{Fa-05}.
In what follows, we briefly describe how these coordinates can be defined in the specific case of system \eqref{eq-HS0}; we refer to \cite{BaFuSa-18, Be-notes} for further details.

Let us assume that there is a non-circular non-rectilinear periodic solution $x^*$ of equation \eqref{eq-03}
which lies on the horizontal plane $x_3 = 0$ and let $(H^*,L^*)$ be the associated values of $\mathcal{H}_0$
and $\mathcal{L}_3$ (notice that $\mathcal{L}_1 = \mathcal{L}_2 = 0$). Let
$\mathcal{S}^*_{(H^*,L^*)}$ be the connected component of the level set
\begin{equation*}
\bigl{\{} (x,p) \in \mathbb{R}^6 \colon  \mathcal{H}_0(x,p) = H^*, \, \vec{\mathcal{L}}(x,p) = (0,0,L^*) \bigr{\}}
\end{equation*}
containing $(x^*(0),p^*(0))$, where as usual $p^*(0) = \varphi(\dot x^*(0))$.
Notice that 
\begin{equation*}
\mathcal{S}^*_{(H^*,L^*)} \subset V_4 \coloneqq \bigl{\{} (x,p) \in \mathbb{R}^6 \colon x_3 = p_3 = 0 \bigr{\}}.
\end{equation*}
By the discussion in Section~\ref{section-2.2}, there are an open set
$\mathcal{D}_{\mathbb{R}^2} \subset \mathbb{R}^2$, an open neighborhood $\mathcal{D}'_{\mathbb{R}^2} \subset \mathbb{R}^2$ of $(H^*,L^*)$, an open neighborhood $\mathcal{U}_{V_4} \subset V_4$ of $\mathcal{S}^*_{(H^*,L^*)}$, a diffeomorphism
\begin{equation*}
\eta \colon \mathcal{D}_{\mathbb{R}^2} \to \mathcal{D}'_{\mathbb{R}^2}, \qquad (I_1,I_2) \mapsto (\eta_1(I_1,I_2),\eta_2(I_1,I_2)), 
\end{equation*}
and a symplectic diffeomorphism
\begin{equation*}
\Psi \colon \mathcal{D}_{\mathbb{R}^2} \times \mathbb{T}^2 \to \mathcal{U}_{V_4}, \qquad 
 (I_1,I_2,\phi_1,\phi_2) \mapsto (x,p),
\end{equation*}
satisfying, for every $I = (I_1,I_2) \in \mathcal{D}_{\mathbb{R}^2}$,
\begin{equation*}
\Psi \bigl( \{ I \} \times \mathbb{T}^2 \bigr) =  \big{\{} (x,p) \in V_4 \colon \mathcal{H}_0(x,p) = \eta_1(I), \mathcal{L}_3(x,p) = \eta_2(I) \bigr{\}} \cap \mathcal{U}_{V_4},
\end{equation*}
and  
transforming system \eqref{eq-HS0} into the form \eqref{action-angle} on the invariant subspace $V_4 \subset \mathbb{R}^6$.
As this point, the transformation $\Psi$ is extended to a symplectic diffeomorphism
\begin{equation}\label{change-variable-3}
\widetilde{\Psi} \colon \mathcal{D}_{\mathbb{R}^2} \times \mathbb{T}^2 \times \mathcal{E}_{\mathbb{R}^2}  \to \mathcal{U}_{\mathbb{R}^6},
\qquad (I_1,I_2,\phi_1,\phi_2,\Xi,\xi) \mapsto (x,p),
\end{equation}
where $\mathcal{U}_{\mathbb{R}^6} \subset \mathbb{R}^6$ is an open neighborhood of $\mathcal{U}_{V_4}$ and $\mathcal{E}_{\mathbb{R}^2} \subset \mathbb{R}^2$ is an open set, in such a way that system \eqref{eq-HS0}
takes the form
\begin{equation}\label{action-angle-3}
\begin{cases}
\, \dot I_i = 0, \qquad \dot \phi_i = \partial_{I_i} \mathcal{K}_0(I_1,I_2), \qquad i=1,2, \vspace{2pt}\\
\, \dot \Xi = 0, \qquad \dot \xi = 0,
\end{cases}
\end{equation}
where $\mathcal{K}_0 \colon \mathcal{D}_{\mathbb{R}^2} \to \mathbb{R}$ is as in formula \eqref{action-angle}. 
The novelty with respect to the integrable case is the presence of the pair of conjugate variables $(\Xi,\xi)$ which, by \eqref{action-angle-3}, remain both constant along the flow.
Without entering into the details, given $(x,p)\in\mathcal{U}_{\mathbb{R}^6}$, the variables $(\Xi,\xi)$ are provided by the plane on which the solution of \eqref{eq-03} with initial conditions $(x(0),\dot{x}(0))=(x,\varphi^{-1}(p))$ lies, see \cite[p.~42]{Be-notes}.
Next, on this plane one can repeat the same construction made on $V_4$ to define the action-angle coordinates $(I_1,I_2,\phi_1,\phi_2)$.
We stress that these partial action-angle coordinates $(I_1,I_2,\phi_1,\phi_2,\Xi,\xi)$ are local in nature, and typically cannot be constructed globally, cf.~Remark~\ref{remark-vario}.

Our next aim is to describe the topology of the manifold
\begin{equation*}
\mathcal{M}_3 = \bigl{\{} M x^*(t-\theta) \colon M\in O(3), \theta \in \mathbb{R} \bigr{\}},
\end{equation*}
composed by all the periodic solutions of \eqref{eq-03} which can be obtained from $x^{*}$ taking into account the invariance by isometries and time-translations.
Notice that $\mathcal{M}_3$ is homeomorphic to the connected component of the level set
\begin{equation*}
\bigl{\{} (x,p) \in \mathbb{R}^6 \colon  \mathcal{H}_0(x,p) = H^*, \, |\vec{\mathcal{L}}(x,p)| = |L^*| \bigr{\}}
\end{equation*}
containing the value $(x^*(0),p^*(0))$, where $p^*(0) = \varphi(\dot x^*(0))$.

\begin{proposition}\label{prop-2.1}
The manifold $\mathcal{M}_3$ is homeomorphic to $SO(3) \times \mathbb{T}^{1}$. In particular, 
\begin{equation*}
\operatorname{dim}(\mathcal{M}_3) = 4  \quad \text{ and } \quad \operatorname{cat}(\mathcal{M}_3) = 5.
\end{equation*}
\end{proposition}

\begin{proof}
For convenience, we suppose that the solution $x^*$ of \eqref{eq-03}, which lies in $x_{3}=0$, is the only one that satisfies also $x^*(0)=(\max |x^*|, 0, 0)$ and $L^*>0$ (by the conservation of energy, $p^*(0)$ is thus uniquely determined). Moreover, we denote again by $\tau^*$ the minimal period of $|x^*|$.

Let us consider the continuous map
\begin{align*}
\Psi \colon SO(3) \times \mathbb{R}/\tau^*\mathbb{Z} &\to \mathcal{M}_3
\\
(M,\theta) &\mapsto M x^{*}(\cdot-\theta).
\end{align*}
Since $SO(3)$ and $\mathbb{R}/\tau^*\mathbb{Z}$ are compact Hausdorff spaces, we just need to show that $\Psi$ is bijective.

We first prove that $\Psi$ is onto. 
Given $x\in\mathcal{M}_3$, let $\theta$ be the first non-negative time in which $|x(\theta)| = \max |x| = \max |x^*|$. 
Moreover, let
\begin{equation*}
M = \operatorname{column} \left( \dfrac{x(\theta)}{|x(\theta)|}, \dfrac{p(\theta)}{|p(\theta)|}, \dfrac{x(\theta) \wedge p(\theta)}{|x(\theta)\wedge p(\theta)|}\right),
\end{equation*}
with $p = \varphi(\dot{x})$. 
It is straightforward to check that $M x^{*}(t-\theta) = x(t)$ for all $t$, due to the uniqueness of the solutions of the associated Cauchy problems.

Next we prove that $\Psi$ is injective.
If $M_{1} x^{*}(t-\theta_{1}) = M_{2} x^{*}(t-\theta_{2})$ for all $t$, then $M_{2}^{-1} M_{1} x^{*}(t-\theta_{1}+\theta_{2}) = x^{*}(t)$ for all $t$, where $M_{2}^{-1} M_{1}\in SO(3)$.
As a consequence $-\theta_{1}+\theta_{2} = k \tau^*$, then $\theta_{1}=\theta_{2}$ (mod. $\tau^*$) and $M_{2}^{-1}M_{1} x^{*}(t) = x^{*}(t)$ for all $t$.
Recalling condition \ref{H1}, we deduce also that 
$M_{2}^{-1}M_{1} p^{*}(t)=\varphi(M_{2}^{-1}M_{1} \dot{x}^{*}(t)) = p^{*}(t)$
for all $t$. The two equations $M_{2}^{-1}M_{1} x^{*}(0) = x^{*}(0)$ and $M_{2}^{-1}M_{1} p^{*}(0)= p^{*}(0)$ are enough to show that $M_{2}^{-1} M_{1} = I_{3}$ since $M_{2}^{-1}M_{1}\in SO(3)$.

By \cite[Theorem~2]{Mo-83} (see also \cite{Si-79}) if $X$ is a closed (piecewise linear) manifold with $\operatorname{cat}(X) \geq (\operatorname{dim}(X)+5)/2$, then
$\operatorname{cat}(X\times \mathbb{T}^{1}) = \operatorname{cat}(X) + 1$.
This result can be applied when $X=SO(3)$, since $\operatorname{dim}(SO(3))=3$
and, by \cite[Lemma~4]{BA-23}, $\operatorname{cat}(SO(3))=4$, giving the thesis.
\end{proof}

\begin{remark}\label{remark-vario}
Note that the periodic manifold $\mathcal{M}_3$ contains two families of solutions which lie on the plane $x_3=0$: the time-translations and planar rotations of $x^*$ (namely the manifold $\mathcal{M}_2$ defined in \eqref{def-M2})
as well as the time-translations and planar rotations of $x^*(-t)$ (which form a homeomorphic copy of $\mathcal{M}_2$).
The union of these two families is a disconnected set of planar solutions which is however connected as a set of spatial solutions
(cf.~\cite[Remark~2.5]{BoFePa-25} in the case of circular solutions).

We also emphasize that, since $\mathcal{M}_3$ has nontrivial topology, the change of variable \eqref{change-variable-3} cannot be global.
\hfill$\lhd$
\end{remark}

\section{Variational formulation and non-degeneracy}\label{section-3}

In this section we provide some further auxiliary tools and results to be used in Section~\ref{section-4}. Precisely, dealing with a general Hamiltonian system in $\mathbb{R}^{2N}$, in Section~\ref{section-3.1} we discuss the variational formulation of both the fixed-period and the fixed-energy problem. Then, for both these problems, in Section~\ref{section-3.2} we give an overview of the notion of non-degenerate periodic manifold.

\subsection{Hamiltonian action functionals}\label{section-3.1}

Let us consider the Hamiltonian system 
\begin{equation}\label{eq-hs-general}
Jz' = \nabla_z \mathcal{H}(t,z), \qquad z \in \mathbb{R}^{2N},
\end{equation}
where 
\begin{equation}
J = \begin{pmatrix}
0 & -I_N \\
I_N & 0
\end{pmatrix}
\end{equation}
is the standard symplectic matrix (here $I_N$ stands for the identity matrix in $\mathbb{R}^N$) and $\mathcal{H} \colon \mathbb{R} \times \mathbb{R}^{2N} \to \mathbb{R}$ is a function which is $T$-periodic in the first variable, for some $T > 0$. We also suppose that 
$\mathcal{H}$ is two times differentiable with respect to $z$, 
and that the functions $\mathcal{H},\nabla_z \mathcal{H}, \nabla_z^2 \mathcal{H}$ are continuous on  $\mathbb{R} \times \mathbb{R}^{2N}$ and have at most polynomial growth with respect to $z$.

As carefully discussed in \cite{Ab-01,BaSz-05}, under these assumptions the $T$-periodic problem for \eqref{eq-hs-general} admits a variational formulation in the fractional Sobolev space
\begin{equation*}
H^{1/2}_T = \biggl\{ z = (x,p) \in L^2 ( \mathbb{R}/T\mathbb{Z}, \mathbb{R}^{2d}) \colon 
\sum_{k \in \mathbb{Z}} \vert k \vert \vert \hat{z}_k \vert^2 < +\infty\biggr\},
\end{equation*}
where $\hat{z}_k \in \mathbb{C}^{2d}$ are the complex Fourier coefficients of the function $z$.
Indeed, the bilinear form
\begin{equation*}
\ell_T(z,w) = \int_0^T \langle J z'(t), w(t) \rangle \,\mathrm{d}t,
\end{equation*}
defined for $\mathcal{C}^1$ and $T$-periodic functions $z,w$, admits, by density, a unique continuous extension to a bilinear form on $H^{1/2}_T$, still denoted by  $\ell_T$.
Moreover, by the assumptions on the Hamiltonian $\mathcal{H}$, the nonlinear term $z \mapsto \int_0^T \mathcal{H}(t,z(t))\,\mathrm{d}t$ is well defined and of class $\mathcal{C}^2$ on $H^{1/2}_T$. Then, the (Hamiltonian) action functional $\mathcal{A}\colon H^{1/2}_T \to \mathbb{R}$ given by
\begin{equation}\label{def-actionA}
\mathcal{A}(z) = \frac{1}{2}  \ell_T(z,z) - \int_0^T \mathcal{H}(t,z(t))\,\mathrm{d}t
\end{equation}
is of class $\mathcal{C}^2$ on the Hilbert space $H^{1/2}_T$. Moreover, it is easy to verify that its critical points correspond to $T$-periodic solutions of \eqref{eq-hs-general} in the classical sense. For further convenience, we also observe that the second differential of $\mathcal{A}$ is given by
the bilinear form 
\begin{equation}\label{def-diff2A}
\mathrm{d}^2 \mathcal{A}(z)[w,u] = \ell_T(w,u) - \int_0^T \langle \nabla_z^2\mathcal{H}(t,z(t))w(t), u(t) \rangle \,\mathrm{d}t,
\end{equation}
for $w,u \in H^{1/2}_T$. Moreover, as proved in \cite[pp.~68--69]{Ab-01}, $\mathrm{d}^2 \mathcal{A}(z)$ is a Fredholm operator of index zero, meaning that its self-adjoint realization $F$, namely
\begin{equation*}
\mathrm{d}^2 \mathcal{A}(z)[w,u] = \langle F w,u \rangle_{H^{1/2}_T},
\end{equation*}
is a linear Fredholm operator on $H^{1/2}_T$ of index zero.

When the Hamiltonian does not depend on time, the same functional setting just described allows us to obtain a variational formulation for the fixed-energy periodic problem
\begin{equation}\label{eq-hs-energy}
\begin{cases}
\, Jz' = \nabla \mathcal{H}(z), \\
\, \mathcal{H}(z) = h.
\end{cases}
\end{equation}
Indeed, moving from the observation that, if $z$ is smooth and $T$-periodic, the function
$\zeta(s) = z(Ts)$ is $1$-periodic and
\begin{equation*}
\frac{1}{2} \int_0^T \langle J z'(t), z(t) \rangle \,\mathrm{d}t  - \int_0^T \mathcal{H}(z(t))\,\mathrm{d}t = 
\frac{1}{2} \int_0^1 \langle J \zeta'(s), \zeta(s) \rangle \,\mathrm{d}s  - T\int_0^1 \mathcal{H}(\zeta(s))\,\mathrm{d}s,
\end{equation*}
one can define the so-called \textit{free-period action functional} $\mathcal{B}\colon H^{1/2}_1 \times (0,+\infty) \to \mathbb{R}$ 
as
\begin{equation}\label{def-actionB}
\mathcal{B}(\zeta,T) = \frac{1}{2}  \ell_1(\zeta,\zeta) - T \int_0^1 \bigl(\mathcal{H}(\zeta(s)) - h \bigr)\,\mathrm{d}s.
\end{equation}
Clearly, $\mathcal{B}$ is of class $\mathcal{C}^2$ on $H^{1/2}_1 \times (0,+\infty)$. Moreover, a critical point $(\zeta,T) \in H^{1/2}_1 \times (0,+\infty)$ gives rise to a $T$-periodic solution of \eqref{eq-hs-energy}. Indeed, since
\begin{equation*}
\mathrm{d}_\zeta \mathcal{B}(\zeta,T)[\mu] = \ell_1(\zeta,\mu) - T \int_0^1 \langle \nabla\mathcal{H}(\zeta(s)), \mu(s) \rangle\,\mathrm{d}s,
\end{equation*}
for every $\mu \in H^{1/2}_1$, the function $\zeta$ is a $1$-periodic solution of $J \zeta' = T \nabla\mathcal{H}(\zeta)$ and, thus, 
$z(t) = \zeta(t/T)$ is a $T$-periodic solution of $J z' = \nabla\mathcal{H}(z)$. On the other hand, since
\begin{equation*}
\mathrm{d}_T \mathcal{B}(\zeta,T)[\sigma] = - \sigma \int_0^1 \bigl(\mathcal{H}(\zeta(s))- h \bigr)\,\mathrm{d}s, 
\end{equation*}
for every $\sigma \in \mathbb{R}$, together with the fact that $\mathcal{H}(\zeta(s))$ is constant since $\zeta$ is a solution, we get that $\mathcal{H}(z(t)) = \mathcal{H}(\zeta(s)) \equiv h$.

For further convenience, we also write down below the expressions of the second differential of $\mathcal{B}$. Precisely, for every $\mu,\nu \in H^{1/2}_1$ and $\sigma,\tau \in \mathbb{R}$,
\begin{align*}
\mathrm{d}^2_{\zeta \zeta} \mathcal{B}(\zeta,T)[\mu,\nu] & = \ell_1(\mu,\nu) - T \int_0^1 \langle \nabla^2\mathcal{H}(\zeta(s))\mu(s), \nu(s) \rangle\,\mathrm{d}s, \\
\mathrm{d}^2_{\zeta T} \mathcal{B}(\zeta,T)[\mu,\sigma] & = - \sigma \int_0^1 \langle \nabla\mathcal{H}(\zeta(s)), \mu(s) \rangle\,\mathrm{d}s, \\
\mathrm{d}^2_{T T} \mathcal{B}(\zeta,T)[\sigma,\tau] & = 0,
\end{align*}
implying that
\begin{align}\label{def-diff2B}
\mathrm{d}^2 \mathcal{B}(\zeta,T)[(\mu,\sigma),(\nu,\tau)] & = \ell_1(\mu,\nu) - T \int_0^1 \langle \nabla^2\mathcal{H}(\zeta(s))\mu(s), \nu(s) \rangle\,\mathrm{d}s \\
& \quad - \tau \int_0^1 \langle \nabla\mathcal{H}(\zeta(s)), \mu(s) \rangle\,\mathrm{d}s - \sigma \int_0^1 \langle \nabla\mathcal{H}(\zeta(s)), \nu(s) \rangle\,\mathrm{d}s.
\end{align}
This is a Fredholm operator of index zero. Indeed, let 
\begin{equation*}
S=(S_1,S_2) \colon H^{1/2}_{1}\times\mathbb{R} \to H^{1/2}_{1}\times\mathbb{R}
\end{equation*}
be the linear continuous self-adjoint operator such that
\begin{align*}
\mathrm{d}^2 \mathcal{B}(\zeta,T)[(\mu,\sigma),(\nu,\tau)] 
&= \langle S(\mu,\sigma), (\nu,\tau)\rangle_{H^{1/2}_{1}\times\mathbb{R}} 
\\
&= \langle S_1(\mu,\sigma), \nu\rangle_{H^{1/2}_{1}} + S_2(\mu,\sigma)\tau.
\end{align*}
Comparing with \eqref{def-diff2B} it is apparent that
\begin{equation*}
S(\mu,\sigma) 
=
\begin{pmatrix}
S_1(\mu,\sigma)\\
S_2(\mu,\sigma)
\end{pmatrix}
= 
\begin{pmatrix}
F & G \\
K & 0
\end{pmatrix}
\begin{pmatrix}
\mu \\
\sigma
\end{pmatrix},
\end{equation*}
where
\begin{align*}
&F \colon H^{1/2}_{1} \to H^{1/2}_{1}, &&
\langle F \mu, \nu \rangle_{H^{1/2}_{1}}
= \ell_1(\mu,\nu) - T \langle \nabla^2\mathcal{H}(\zeta)\mu, \nu \rangle_{L^2},
\\
&G \colon \mathbb{R} \to H^{1/2}_{1}, &&\langle G\sigma, \nu \rangle_{H^{1/2}_{1}} = - \sigma \langle \nabla \mathcal{H}(\zeta),\nu\rangle_{L^2},
\\
&K \colon H^{1/2}_{1} \to \mathbb{R}, &&K\mu = - \langle \nabla \mathcal{H}(\zeta),\mu\rangle_{L^2}.
\end{align*}
Since $G$ and $K$ are compact, $S$ is a compact perturbation of the operator
\begin{equation*}
\widetilde{F} \colon H^{1/2}_{1} \times \mathbb{R} \to H^{1/2}_{1} \times \mathbb{R},
\quad
\widetilde{F}(\mu,\sigma) 
=
\begin{pmatrix}
F & 0 \\
0 & 0
\end{pmatrix}
\begin{pmatrix}
\mu \\
\sigma
\end{pmatrix},
\end{equation*}
which is Fredholm of index zero, since $F$ is Fredholm of index zero (by the result for the fixed-period case) and
\begin{equation*}
\ker \widetilde{F} = \ker F \times \mathbb{R},
\qquad
\mathrm{Im} \widetilde{F} = \mathrm{Im} F \times \{0\}.
\end{equation*}
Accordingly, by the stability of the Fredholm index with respect to compact perturbations (see, for instance, \cite[Chapter~IV, \S~5]{Ka-95}), $S$ is Fredholm of index zero, as well.

\subsection{Non-degenerate periodic manifolds}\label{section-3.2}

Let us now consider the autonomous Hamiltonian system 
\begin{equation}\label{eq-hs-general-3.2}
Jz' = \nabla \mathcal{H}(z), \qquad z \in \Omega \subset \mathbb{R}^{2N},
\end{equation}
and assume that $\mathfrak{M} \subset \Omega$ is a compact manifold made up by (initial conditions of) non-constant periodic solutions of 
\eqref{eq-hs-general-3.2} (with a slight abuse of notation, in what follows we systematically identify any $z \in \mathfrak{M}$ with the solution $z(t)$ with $z(0) = z$).  Moreover, we suppose that a smooth choice $\mathfrak{M} \ni z \mapsto T(z)$ can be made for a period of $z$; accordingly, we denote by
\begin{equation*}
P_{z} \colon \mathbb{R}^{2N} \to \mathbb{R}^{2N}
\end{equation*}
the Poincar\'e operator at time $T(z)$ for the linear system $Jw'= \nabla^2 \mathcal{H}(z(t))w$ (that is, the so-called monodromy of $z$).

We start by discussing the (rather standard) notion of non-degeneracy for the fixed-period problem. In the following proposition, we use the action functional $\mathcal{A}$ defined in \eqref{def-actionA} (with $\mathcal{H}(t,z) = \mathcal{H}(z)$): note that in principle the definition is not well-posed, since $\mathcal{H}$ is defined only on $\Omega \subset \mathbb{R}^{2N}$ and 
$H^{1/2}_T$-functions are not bounded in general. In spite of this, we can formally use, for every $z \in \mathfrak{M}$, the second differential $\mathrm{d}^2 \mathcal{A}(z)$, meant as the bilinear form defined in \eqref{def-diff2A}.

\begin{proposition}\label{prop-nondegper}
Assume that $T(z) = T$ for every $z \in \mathfrak{M}$. Then, the following conditions are equivalent:
\begin{itemize}
\item[$(i)$] for every $z \in \mathfrak{M}$, the dimension of $\ker ( I - P_z )$ equals the dimension of $\mathfrak{M}$;
\item[$(ii)$] for every $z \in \mathfrak{M}$, the dimension of the space of $T$-periodic solutions of the linear system
\begin{equation*}
Jw' = \nabla^2 \mathcal{H}(z(t)) w
\end{equation*}
equals the dimension of $\mathfrak{M}$;
\item[$(iii)$] for every $z \in \mathfrak{M}$, the dimension of 
$\ker ( \mathrm{d}^2 \mathcal{A}(z))$ equals the dimension of $\mathfrak{M}$.
\end{itemize}
Moreover, if system \eqref{eq-hs-general-3.2} has the form \eqref{eq-HS0} with $d=2$ 
and $\mathfrak{M} = \Pi(\mathcal{M}_2)$, with $\mathcal{M}_2$ and $\Pi$ defined in \eqref{def-M2} and \eqref{eq-projection} respectively, then the above conditions are equivalent to
\begin{itemize}
\item[$(iv)$] it holds that 
\begin{equation*}
\det \nabla^2 \mathcal{K}_0(I^*) \neq 0,
\end{equation*} 
where $\mathcal{K}_0$ is the Hamiltonian in action-angle coordinates (see \eqref{action-angle}) and $I^*$ is the action corresponding to $\mathcal{M}_2$ via \eqref{def-Istar}.
\end{itemize}
\end{proposition}

According to the above proposition, a manifold $\mathfrak{M}$ is called \textit{non-degenerate for the fixed-period problem} if it satisfies one (and, hence, all) of the conditions $(i)$, $(ii)$ and $(iii)$ therein.

\begin{proof}[Proof of Proposition~\ref{prop-nondegper}]
The equivalence between $(i)$, $(ii)$ and $(iii)$ is straightforward. As for $(iv)$,
we observe that the linearization of \eqref{action-angle} along any solution of $\mathfrak{M}= \Pi(\mathcal{M}_2)$ is given by
\begin{equation*}
\dot X = 0, \qquad \dot Y = \nabla^2 \mathcal{K}_{0}(I^*)X, \qquad (X,Y) \in \mathbb{R}^2 \times \mathbb{R}^2.
\end{equation*}
Thus the monodromy $Q\colon \mathbb{R}^4 \to \mathbb{R}^4$ is given by
\begin{equation*}
Q(X,Y) = (X, T\nabla^2 \mathcal{K}_{0}(I^*)X + Y)
\end{equation*}
and therefore
\begin{equation*}
\operatorname{dim} ( \ker ( I -Q ) ) = 2 + \operatorname{nullity} \bigl( \nabla^2 \mathcal{K}_{0}(I^*)\bigr).
\end{equation*}
Since $\Pi(\mathcal{M}_2)$ has dimension two, we obtain
\begin{equation*}
\operatorname{dim} \,( \ker ( I -Q ) ) = \operatorname{dim} \,(\Pi(\mathcal{M}_2) ) \quad \Longleftrightarrow \quad \det \nabla^2 \mathcal{K}_0(I^*) \neq 0.
\end{equation*}
Recalling that $\ker (I-Q)$ and $\ker ( I - P_z)$ are isomorphic (see, for instance, \cite[Appendix~A]{BoFePa-25}), we infer the equivalence between $(iv)$ and $(i)$.
\end{proof}

We now turn our attention to the less popular fixed-energy problem. Here, we formally use the free-period action functional $\mathcal{B}$ defined in \eqref{def-actionB} and its second differential $\mathrm{d}^2 \mathcal{B}$ as computed in \eqref{def-diff2B} (the same considerations given before Proposition~\ref{prop-nondegper} about the well-posedness of these definitions can be repeated here).

\begin{proposition}\label{prop-nondegene}
Assume that $\mathfrak{M} \subset \mathcal{H}^{-1}(h)$ for some $h \in \mathbb{R}$. Then, the following conditions are equivalent:
\begin{itemize}
\item[$(i)$] for every $z \in \mathfrak{M}$, the dimension of the linear space 
\begin{equation*}
\mathcal{F} = \bigl{\{} w \in T_z (\mathcal{H}^{-1}(h)) \colon w = P_z w + \lambda J\nabla \mathcal{H}(z), \text{ for some $\lambda \in \mathbb{R}$} \bigr{\}}
\end{equation*}
equals the dimension of $\mathfrak{M}$ (here $T_z$ means the tangent space); 
\item[$(ii)$] for every $z \in \mathfrak{M}$, the dimension of the linear space of $T(z)$-periodic solutions of the problem
\begin{equation}\label{linear-syst-fix-en}
\begin{cases}
\, Jw' = \nabla^2 \mathcal{H}(z(t)) w + \gamma \nabla \mathcal{H}(z(t)), \; \text{ for some $\gamma \in \mathbb{R}$,}
\\
\, \langle \nabla\mathcal{H}(z(t)),w \rangle \equiv 0,
\end{cases}
\end{equation}
equals the dimension of $\mathfrak{M}$;
\item[$(iii)$] for every $z \in \mathfrak{M}$, the dimension of 
$\ker ( \mathrm{d}^2 \mathcal{B}(\zeta,T(z)) )$ equals the dimension of $\mathfrak{M}$, where 
$\zeta(s) = z(T(z)s)$.
\end{itemize}
Moreover, if system \eqref{eq-hs-general-3.2} has the form \eqref{eq-HS0} with $d=2$ 
and $\mathfrak{M} = \Pi(\mathcal{M}_2)$, with $\mathcal{M}_2$ and $\Pi$ defined in \eqref{def-M2} and \eqref{eq-projection} respectively, then the above conditions are equivalent to
\begin{itemize}
\item[$(iv)$] it holds that
\begin{equation*}
\operatorname{det}
\begin{pmatrix}
\nabla^2 \mathcal{K}_{0}(I^{*}) & \nabla \mathcal{K}_{0}(I^{*})^{\top} \vspace{3pt} \\
\nabla \mathcal{K}_{0}(I^{*}) & 0
\end{pmatrix} 
\neq 0,
\end{equation*}
where $\mathcal{K}_0$ is the Hamiltonian in action-angle coordinates (see \eqref{action-angle}) and $I^*$ is the action corresponding to $\mathcal{M}_2$ via \eqref{def-Istar}.
\end{itemize}
\end{proposition}

According to the above proposition, a manifold $\mathfrak{M}$ is called \textit{non-degenerate for the fixed-energy problem} if it satisfies one (and, hence, all) of the conditions $(i)$, $(ii)$ and $(iii)$ therein. 
The equivalence between these three conditions has been already discussed in \cite{We-78}; however for the reader's convenience we review these arguments in the proof below.

\begin{proof}[Proof of Proposition~\ref{prop-nondegene}]
In order to prove that $(i)$ and $(ii)$ are equivalent, we show that $w$ is a $T(z)$-periodic solution of \eqref{linear-syst-fix-en} if and only if $w(0)\in\mathcal{F}$. In fact, if $w$ is a $T(z)$-periodic solution of \eqref{linear-syst-fix-en}, then
$\widetilde{w}(t) \coloneqq w(t) + \gamma t J \nabla \mathcal{H}(z(t))$ satisfies $J \widetilde{w}' = \nabla^{2}\mathcal{H}(z(t)) \widetilde{w}$, and so $\widetilde{w}(T(z)) = P_z \widetilde{w}(0)$. Moreover, 
\begin{equation*}
\widetilde{w}(0)=w(0), \quad \widetilde{w}(T(z))=w(T(z))+\gamma T(z) J \nabla \mathcal{H}(z(0)),
\end{equation*}
and
\begin{equation*}
(I-P_z)w(0) = \widetilde{w}(0)-\widetilde{w}(T(z)) = -\gamma T(z) J \nabla \mathcal{H}(z(0)).
\end{equation*}
Furthermore, $\langle \nabla \mathcal{H}(z(0)), w(0) \rangle =0$, finally proving that $w(0)\in \mathcal{F}$.
On the other hand, if $w_{0}\in\mathcal{F}$ and $\widetilde{w}$ is the solution of the Cauchy problem
\begin{equation*}
\begin{cases}
\, J \widetilde{w}' = \nabla^{2} \mathcal{H}(z(t)) \widetilde{w},
\\
\, \widetilde{w}(0) = w_{0},
\end{cases}
\end{equation*}
the function $w(t) \coloneqq \widetilde{w}(t)- \lambda t J \nabla \mathcal{H}(z(t))$ satisfies
\begin{equation*}
J w' = \nabla^{2} \mathcal{H}(z(t)) w + \lambda\nabla \mathcal{H}(z(t)),
\end{equation*}
and 
\begin{equation*}
w(T(z)) = \widetilde{w}(T(z)) - \lambda T(z) J \nabla \mathcal{H}(z(T(z)))
= P_z w_{0} - \lambda T(z) J \nabla \mathcal{H}(z(0))
= w_{0} = w(0).
\end{equation*}
Straightforward computations show that
\begin{equation}\label{eq-dtH}
\dfrac{\mathrm{d}}{\mathrm{d}t} \langle \nabla\mathcal{H}(z(t)), w(t) \rangle = 0,
\end{equation}
hence $\langle \nabla\mathcal{H}(z(t)), w(t) \rangle = \langle \nabla\mathcal{H}(z(0)), w_0 \rangle = 0$ for all $t$. We conclude that $w$ is a $T(z)$-periodic solution of \eqref{linear-syst-fix-en}.

We now show the equivalence between $(ii)$ and $(iii)$. Let $\mu\in H^{1/2}_{1}$ and $\sigma\in\mathbb{R}$ be such that
$(\mu,\sigma)\in\ker ( \mathrm{d}^2 \mathcal{B}(\zeta,T(z)) )$, that is 
\begin{equation}\label{eq-ker-d2}
\mathrm{d}^2 \mathcal{B}(\zeta,T(z))[(\mu,\sigma),(\nu,\tau)] = 0
\end{equation}
for every $\nu\in H^{1/2}_{1}$ and $\tau\in\mathbb{R}$, cf.~\eqref{def-diff2B}.
In particular, taking $\tau=0$, we find that
\begin{equation*} 
\ell_1(\mu,\nu) - T(z) \int_0^1 \langle \nabla^2\mathcal{H}(\zeta(s))\mu(s), \nu(s) \rangle\,\mathrm{d}s \\
- \sigma \int_0^1 \langle \nabla\mathcal{H}(\zeta(s)), \nu(s) \rangle\,\mathrm{d}s = 0
\end{equation*}
for every $\nu\in H^{1/2}_{1}$, which means that $\mu$ is a $1$-periodic solution of
\begin{equation*} 
J\mu' = T(z) \nabla^{2} \mathcal{H}(z(s))\mu + \sigma \nabla\mathcal{H}(z(s)).
\end{equation*}
Therefore, $w(t) \coloneqq \mu(t/T(z))$ is a $T(z)$-periodic solution of the differential equation in \eqref{linear-syst-fix-en} with $\gamma=\sigma/T(z)$. Next, taking $\nu=0$ in \eqref{eq-ker-d2}, we find that
\begin{equation*} 
0 = \int_0^1 \langle \nabla\mathcal{H}(\zeta(s)), \mu(s) \rangle\,\mathrm{d}s = \int_0^{T(z)} \langle \nabla\mathcal{H}(z(t)), w(t) \rangle\,\mathrm{d}t.
\end{equation*}
Observing that the integrand is constant (cf.~\eqref{eq-dtH}), we infer that $w$ satisfies the second equation in \eqref{linear-syst-fix-en} too. 
On the other hand, if $w$ is a $T(z)$-periodic solution of \eqref{linear-syst-fix-en} for some $\gamma \in \mathbb{R}$, then the same argument shows that $(\mu(s),\sigma)=(w(T(z) s), T(z) \gamma)$ belongs to $\ker ( \mathrm{d}^2 \mathcal{B}(\zeta,T(z)))$.
Summing up, $(ii)$ and $(iii)$ are equivalent.

Let us finally deal with $(iv)$. As already observed in the proof of Proposition~\ref{prop-nondegper}, the linearization of \eqref{action-angle} along a solution of $\mathfrak{M} = \Pi(\mathcal{M}_2)$ (note that the period is constant along the manifold) yields the monodromy
\begin{equation*}
Q(X,Y) = (X, T \nabla^2 \mathcal{K}_0(I^*)X + Y).
\end{equation*}
Let us consider the linear space
\begin{equation*}
\mathcal{G} = \bigl{\{} (X,Y) \in T_{(I^*,\phi)}\mathcal{K}_0^{-1}(h) \colon (X,Y) = Q(X,Y) + \lambda (0,\nabla \mathcal{K}_0(I^*)), \text{ for some $\lambda \in \mathbb{R}$} \bigr{\}},
\end{equation*}
where (meaning, with a slight abuse, $\mathcal{K}_0$ as a function of both $I$ and $\phi$)
\begin{equation*}
T_{(I^*,\phi)}\mathcal{K}_0^{-1}(h) = \bigl{\{} (X,Y) \in \mathbb{R}^2 \times \mathbb{R}^2 \colon \langle \nabla 
\mathcal{K}_0(I^*),X\rangle = 0 \bigr{\}}.
\end{equation*}
Recalling that $\Pi(\mathcal{M}_2)$ has dimension two, we claim that 
\begin{equation*}
\operatorname{dim} \mathcal{G} = \operatorname{dim} \, (\Pi(\mathcal{M}_2)) \quad \Longleftrightarrow \quad \operatorname{det}
\begin{pmatrix}
\nabla^2 \mathcal{K}_{0}(I^{*}) & \nabla \mathcal{K}_{0}(I^{*})^{\top} \vspace{3pt} \\
\nabla \mathcal{K}_{0}(I^{*}) & 0
\end{pmatrix} 
\neq 0.
\end{equation*}
Indeed, $\operatorname{dim} \mathcal{G} \geq 2$ since $(0,Y) \in \mathcal{G}$ for every $Y \in \mathbb{R}^2$; 
hence, $\dim \mathcal{G} > 2$ if and only if there are $\lambda \in \mathbb{R}$ and $X \in \mathbb{R}^2 \setminus \{0\}$ such that
\begin{equation*}
\langle \nabla \mathcal{K}_0(I^*), X \rangle = 0 \quad \text{ and } \quad \nabla^2 \mathcal{K}_0(I^*) X = \lambda \nabla \mathcal{K}_0(I^*).
\end{equation*}
This is equivalent to say that the $(2+1) \times (2+1)$ homogeneous linear system defined by the matrix
\begin{equation*}
\begin{pmatrix}
\nabla^2 \mathcal{K}_{0}(I^{*}) & \nabla \mathcal{K}_{0}(I^{*})^{\top} \vspace{3pt} \\
\nabla \mathcal{K}_{0}(I^{*}) & 0
\end{pmatrix} 
\end{equation*}
admits non-trivial solutions (note that a non-trivial solution of the form $(\lambda^*,0)$ cannot exist, since 
$\nabla \mathcal{K}_{0}(I^{*}) \neq 0$), from which the claim follows. Since $\mathcal{F}$ and $\mathcal{G}$ are isomorphic (see \cite[Appendix~A]{BoFePa-25}), we infer the equivalence between $(iv)$ and $(i)$.
\end{proof}

\begin{remark}\label{rem-Lagrangian-non-degeneracy}
If the Hamiltonian system \eqref{eq-hs-general-3.2} comes, via Legendre transformation, from a Lagrangian system (as in the case of system \eqref{eq-HS0}), a further equivalent formulation of non-degeneracy can be given in terms of the kernel of the second differential of the Lagrangian action functional
\begin{equation*}
\mathcal{I}(x) = \int_0^T \mathcal{L}(x(t),\dot x(t))\,\mathrm{d}t, \qquad \text{$x$ is $T$-periodic,}
\end{equation*}
and of the free-period Lagrangian action functional
\begin{equation*}
\mathcal{J}(\gamma,T) = T \int_0^1 \bigl(\mathcal{L}(\gamma(s),\dot \gamma(s)/T) - h \bigr)\,\mathrm{d}s, \qquad \text{$\gamma$ is $1$-periodic,}
\end{equation*}
respectively (computations in the fixed-energy case are slightly more delicate, since the functional $\mathcal{J}$ depends in a more involved manner from the variable $T$, see for instance \cite[Appendix~A]{AbMaPa-15}).

We also point out that, for a Lagrangian system of the type $\ddot{x} = \nabla W(x)$, the fixed-energy problem can be formulated as a critical point problem for the Maupertuis functional 
\begin{equation*}
\mathcal{M}(\gamma) = \int_{0}^{1} |\dot \gamma(s)|^2 \, \mathrm{d}s \int_{0}^{1} (W(\gamma(s))+h)\, \mathrm{d}s, \qquad \text{$\gamma$ is $1$-periodic.}
\end{equation*}
Accordingly a further notion of non-degeneracy arises considering the kernel of the second differential of $\mathcal{M}$, see \cite{AmBe-92}: it can be checked that this notion agrees with the previous ones as well.
In the relativistic case, a Maupertuis-type functional has been introduced in \cite{BoDaMH-23}: the equivalence of the associated non-degeneracy conditions could be investigated.
\hfill$\lhd$
\end{remark}

\section{Statement and proof of the main results}\label{section-4}

In this section we state and prove our main results dealing with the equation
\begin{equation}\label{eq-main-bis}
\dfrac{\mathrm{d}}{\mathrm{d}t} \bigl( \varphi(\dot{x}) \bigr) = V'(|x|) \dfrac{x}{|x|} + E_{\varepsilon}(t,x)+\dot{x} \wedge B_{\varepsilon}(t,x), \qquad x \in \mathbb{R}^3 \setminus \{0\}, 
\end{equation}
where, of course, we continue to assume the hypotheses \ref{H1}, \ref{H2} and \ref{H3} listed in the Introduction.
More precisely, in what follows we consider both the $T$-periodic problem 
\begin{equation}\label{eq-main-periodo}
\begin{cases}
\, \dfrac{\mathrm{d}}{\mathrm{d}t} \bigl( \varphi(\dot{x}) \bigr) = V'(|x|) \dfrac{x}{|x|} + E_{\varepsilon}(t,x)+\dot{x} \wedge B_{\varepsilon}(t,x), \qquad x \in \mathbb{R}^3 \setminus \{0\}, 
\vspace{2pt}\\
\, (x(0),x(T)) =  (\dot x(0),\dot x(T)),
\end{cases}
\end{equation}
as well as, in the case when $U_\varepsilon$ and $A_\varepsilon$ do not depend on time,
the fixed-energy periodic problem
\begin{equation}\label{eq-main-energia}
\begin{cases}
\, \dfrac{\mathrm{d}}{\mathrm{d}t} \bigl( \varphi(\dot{x}) \bigr) = V'(|x|) \dfrac{x}{|x|} + E_{\varepsilon}(x)+\dot{x} \wedge B_{\varepsilon}(x), \qquad x \in \mathbb{R}^3 \setminus \{0\}, \vspace{2pt}\\
(x(0),x(T)) =  (\dot x(0),\dot x(T)), 
\vspace{2pt}\\
\, \mathcal{E}_\varepsilon(x,\dot x) \equiv h,
\end{cases}
\end{equation}
where $h \in \mathbb{R}$, the energy $\mathcal{E}_\varepsilon$ is defined as
\begin{equation*}
\mathcal{E}_\varepsilon(x,v) = G(\vert \varphi(v) \vert) - V(\vert x \vert) - U_\varepsilon(x),
\end{equation*}
cf.~\eqref{eq-defenergy}, and $T > 0$ is \textit{unprescribed}. Note that a solution to \eqref{eq-main-energia} is thus a pair $(x,T)$, where $x$ is a periodic solution of the differential equation, with $\mathcal{E}_\varepsilon(x,\dot x ) \equiv h$, and $T$ is a period (not necessarily the minimal one) for $x$.

The plan of the section is the following. In Section~\ref{section-4.1}, we first provide two abstract bifurcation theorems, giving the existence of bifurcating solutions to \eqref{eq-main-periodo} and \eqref{eq-main-energia}, respectively, when the existence of a non-degenerate periodic manifold $\mathcal{M}$ of the unperturbed problem is assumed a priori. 
Then, in Section~\ref{section-4.2}, we finally give our main results, proving bifurcation from the periodic manifold $\mathcal{M}_3$, see \eqref{defM3}, provided that a non-degeneracy condition is assumed for the unperturbed planar problem.

\subsection{Abstract bifurcation theorems}\label{section-4.1}

In the following results, we assume the existence of a compact manifold $\mathcal{M}$ of non-constant non-rectilinear periodic solutions of the unperturbed problem
\begin{equation}\label{eq-unperturbed}
\dfrac{\mathrm{d}}{\mathrm{d}t} \bigl( \varphi(\dot{x}) \bigr) = V'(|x|) \dfrac{x}{|x|}, \qquad x \in \mathbb{R}^3 \setminus\{0\},
\end{equation}
which will be required to have either fixed period (in the case of Theorem~\ref{th-abstract1}) or fixed energy (in the case of Theorem~\ref{th-abstract2}). From now on, with a slight abuse of notation, we systematically identify $\mathcal{M}$ with the manifold $\mathfrak{M}$ of periodic solutions $(x,p) = (x, \varphi(\dot x))$ of the associated Hamiltonian system in $\mathbb{R}^{6}$, see \eqref{eq-HS0}; in particular, we agree to say that $\mathcal{M}$ is non-degenerate in the sense of Proposition~\ref{prop-nondegper} or Proposition~\ref{prop-nondegene} if $\mathfrak{M}$ is. 

Here are the statement of our results, dealing respectively with the fixed-period problem \eqref{eq-main-periodo} and with the fixed-energy problem \eqref{eq-main-energia}.

\begin{theorem}\label{th-abstract1}
Let $\mathcal{M}$ be a compact manifold of non-constant non-rectilinear $T$-periodic solutions for the unperturbed problem \eqref{eq-unperturbed}. Moreover, assume that $\mathcal{M}$ is non-degenerate in the sense of Proposition~\ref{prop-nondegper}.
Then, for every $\sigma > 0$, there exists $\varepsilon^* > 0$ such that for every $\varepsilon \in (-\varepsilon^*,\varepsilon^*)$ there are 
$m \coloneqq \operatorname{cat}(\mathcal{M})$ solutions $x_i$ of \eqref{eq-main-periodo} and $x^*_i \in \mathcal{M}$ ($i=1,\ldots,m$) such that
\begin{equation}\label{cond-th-abstract1}
\vert x_i(t) - x^*_i(t) \vert < \sigma, \quad  \text{for every $t \in [0,T]$,}
\end{equation}
for $i=1,\ldots,m$.
\end{theorem}

\begin{theorem}\label{th-abstract2}
Let $\mathcal{M}$ be a compact manifold of non-constant non-rectilinear periodic solutions for the unperturbed problem \eqref{eq-unperturbed} satisfying, for some $h \in \mathbb{R}$,
\begin{equation*}
\mathcal{E}_0(x(t),\dot x(t)) \equiv h, \quad \text{for every $x \in \mathcal{M}$.}
\end{equation*}
Moreover, assume that a smooth choice $\mathcal{M} \ni x \mapsto T(x)$ can be made for a period of $x$ and that 
$\mathcal{M}$ is non-degenerate in the sense of Proposition~\ref{prop-nondegene}.
Then, for every $\sigma > 0$, there exists $\varepsilon^* > 0$ such that for every $\varepsilon \in (-\varepsilon^*,\varepsilon^*)$ there are a solution $(x,T)$ of \eqref{eq-main-energia} and $x^* \in \mathcal{M}$ such that $\vert T - T(x^*) \vert < \sigma$
and
\begin{equation}\label{cond-th-abstract2}
\vert x(t) - x^*(t) \vert < \sigma, \quad  \text{for every $t \in [0,T]$.}
\end{equation}
\end{theorem}

Conditions \eqref{cond-th-abstract1} and \eqref{cond-th-abstract2} ensure that the solutions found remain, for $\varepsilon \to 0$, arbitrarily close to the periodic manifold $\mathcal{M}$: in this sense, we agree to say that they bifurcate from $\mathcal{M}$
(note that in the fixed-energy case also the periods are close, due to $\vert T - T(x^*) \vert < \sigma$). We point out that, as it will be clear from the proofs, it holds that $\vert \varphi(\dot x(t)) - \varphi(\dot x^*(t)) \vert < \sigma$ for every $t \in [0,T]$, as well.
A couple of further remarks about the statements are now in order.

\begin{remark}\label{rem-cat-multiplicity}
While Theorem~\ref{th-abstract1} provides $m = \operatorname{cat}(\mathcal{M})$ solutions of the fixed-period problem \eqref{eq-main-periodo}, just one solution is obtained for the fixed-energy problem \eqref{eq-main-energia} using Theorem~\ref{th-abstract2}. The reason for this difference lies in the fact that, while the non-autonomous perturbed problem \eqref{eq-main-periodo}
breaks in general all the invariances (rotations, time-translations and time-inversions) of the unperturbed problem, the autonomous perturbed problem \eqref{eq-main-energia} is still invariant by time-translation (and by time-inversion if $B_\varepsilon \equiv 0$): thus, in principle, multiple bifurcating solutions could be in the same equivalence class for such invariances.
A way to overcome this difficulty could be that of working in a suitable quotient space and using an equivariant version of the Lusternik--Schnirelmann category, as done, for instance, in \cite{AmBe-92} (see, in particular, Remark~4 therein). However, the computation of the equivariant category is typically very difficult and, in particular, we do not know a way to compute it in the setting treated in Section~\ref{section-4.2}. For this reason, we have preferred to state Theorem~\ref{th-abstract2} in a simpler form, avoiding at all any multiplicity claim.
\hfill$\lhd$
\end{remark}

\begin{remark}\label{rem-d-dim}
Even if, for the sake of simplicity in the exposition, we have stated Theorem~\ref{th-abstract1} and Theorem~\ref{th-abstract2} for the three-dimensional equation \eqref{eq-main-bis}, the very same results hold for the $d$-dimensional system, with $d\geq2$,
\begin{equation}\label{eq-rem-d}
\dfrac{\mathrm{d}}{\mathrm{d}t} \bigl( \varphi(\dot{x}) \bigr) 
= V'(|x|) \dfrac{x}{|x|} + E_{\varepsilon}(t,x)+ C_{\varepsilon}(t,x) \dot{x}, \qquad x \in \mathbb{R}^d \setminus \{0\}, 
\end{equation}
where $U_\varepsilon\colon \mathbb{R} \times (\mathbb{R}^d \setminus \{0\}) \to \mathbb{R}$, $A_\varepsilon \colon \mathbb{R} \times (\mathbb{R}^d \setminus \{0\}) \to \mathbb{R}^d$, and
\begin{equation*}
E_{\varepsilon}(t,x) = \nabla_{x}U_{\varepsilon}(t,x) - \dfrac{\partial}{\partial t} A_{\varepsilon}(t,x),
\qquad
C_{\varepsilon}(t,x) = (\mathrm{D}_{x}A_{\varepsilon}(t,x))^{\top}- \mathrm{D}_{x}A_{\varepsilon}(t,x).
\end{equation*}
Indeed, it will be clear that the proofs of Theorem~\ref{th-abstract1} and Theorem~\ref{th-abstract2} work for the Hamiltonian \eqref{def-Ham} in any dimension $d\geq2$ and such a Hamiltonian provides the Lagrangian system \eqref{eq-rem-d}.
For $d=3$, due to \eqref{eq-vectorcalculus}, this system coincides with \eqref{eq-main-bis}.
Notice that for $d=2$ we have instead
\begin{equation*}
C_{\varepsilon}(t,x) = (\partial_{x_{2}} A^{1}_{\varepsilon}(t,x) - \partial_{x_{1}} A^{2}_{\varepsilon}(t,x) ) J,
\end{equation*}
with $J$ as in \eqref{eq-standsym}. In particular, taking $A_{\varepsilon}(t,x) = \varepsilon (x_{2},0)$, we obtain an equation of the form
\begin{equation*}
\dfrac{\mathrm{d}}{\mathrm{d}t} \bigl( \varphi(\dot{x}) \bigr) 
= V'(|x|) \dfrac{x}{|x|} + \nabla_{x}U_{\varepsilon}(t,x) + \varepsilon J \dot{x}, \qquad x \in \mathbb{R}^2 \setminus \{0\}.
\end{equation*}
This kind of structure is typical of restricted problems of celestial mechanics when passing to a rotating system of reference (cf.~\cite{So-14}).
\hfill$\lhd$
\end{remark}

The proofs of Theorem~\ref{th-abstract1} and Theorem~\ref{th-abstract2} rely on a variational approach, 
based on the use, respectively, of the Hamiltonian action functional and of the free-period Hamiltonian action functional introduced in Section~\ref{section-3.1}
(the reason to adopt a Hamiltonian formulation instead of the Lagrangian one, see Remark~\ref{rem-Lagrangian-non-degeneracy}, comes from regularity issues, see \cite[Remark~3.3]{BoFePa-25}).
Precisely, we are going to apply an abstract perturbation theorem (see \cite[Theorem~2.1]{AmCoEk-87} and the references therein), which we recall here for the reader's convenience, in the version stated in \cite[Theorem~10.8]{MaWi-89}.

\begin{theorem}\label{thastratto1}
Let $H$ be a real Hilbert space, $\Omega \subset H$ be an open set, and $\mathcal{F}_\varepsilon \colon \Omega \to \mathbb{R}$
be a family of twice continuously differentiable functions depending smoothly on $\varepsilon$.
Moreover, let $\mathfrak{N} \subset \Omega$ be a compact manifold (without boundary) such that:
\begin{itemize}
\item[$(i)$] $\mathrm{d}\mathcal{F}_0(y) = 0$, for every $y \in \mathfrak{N}$;
\item[$(ii)$] $\mathrm{d}^2 \mathcal{F}_0(y)$ is a Fredholm operator of index zero, for every $y \in \mathfrak{N}$;
\item[$(iii)$] $T_y \mathfrak{N} = \ker (\mathrm{d}^2 \mathcal{F}_0(y))$, for every $y \in \mathfrak{N}$.
\end{itemize}
Then, for every neighborhood $U$ of $\mathfrak{N}$, 
there exists $\varepsilon^*> 0$ such that, 
if $\varepsilon\in(-\varepsilon^*,\varepsilon^*)$, the functional
$\mathcal{F}_\varepsilon$ has at least $\operatorname{cat}(\mathfrak{N})$ critical points in $U$.
\end{theorem}

Let us consider equation \eqref{eq-main-bis}, which, as discussed in Section~\ref{section-2.1}, can be written in an equivalent way as the Hamiltonian system associated with the Hamiltonian $\mathcal{H}_{\varepsilon}$ defined in \eqref{def-Ham},
which from now on we write as
\begin{equation*}
\mathcal{H}_{\varepsilon}(t,x,p) = G(|p|) - V(|x|) + \mathcal{R}_{\varepsilon}(t,x,p),
\end{equation*}
with
\begin{equation*}
\mathcal{R}_{\varepsilon}(t,x,p) = - U_{\varepsilon}(t,x) + G(|p-A_{\varepsilon}(t,x)|)-G(|p|).
\end{equation*}
Since $\mathcal{H}_{\varepsilon}$ is singular at $x=0$ and exhibits a possible lack of smoothness at $p=A_{\varepsilon}(t,x)$, 
we are not allowed to write down the corresponding action functionals as described in Section~\ref{section-3.1}:
the fact that we can work locally around a compact manifold $\mathfrak{N}$ is, in itself, not enough to remedy the difficulty, since functions in the Hilbert space $H^{1/2}$ are not bounded, in general.

To overcome this issue, the strategy will be that of defining a modified Hamiltonian $\widehat{\mathcal{H}}_{\varepsilon}$ and then checking, with an ad-hoc argument, that periodic solutions of the modified system also solve the original one. 
More precisely, we argue as follows. 
First, we observe that, since the compact manifold $\mathcal{M}$ does not contain rectilinear solutions of \eqref{eq-unperturbed}, then
\begin{equation*}
x(t) \neq 0 \quad \text{ and } \quad \varphi(\dot x(t)) \neq 0, \quad \text{for every $t \in \mathbb{R}$,}
\end{equation*}
for every $x \in \mathcal{M}$
and so, by the compactness of $\mathcal{M}$, there are $r_1,r_2,\rho_1,\rho_2>0$ such that
\begin{equation}\label{def-r-rho}
0 < r_1 \leq |x(t)| \leq r_2, 
\quad
0 < \rho_1 \leq |\varphi(\dot x(t))| \leq \rho_2, \quad \text{for every $t \in \mathbb{R}$.}
\end{equation}
Accordingly, we consider two smooth functions $\chi_1,\chi_2\colon[0,+\infty) \to [0,+\infty)$ such that
\begin{equation*}
\chi_1(s)=
\begin{cases}
\, 1, &\text{if } s\in\bigl[\frac{r_1}{2},2r_2\bigr],
\\
\, 0, &\text{if } s\in \bigl[0,\frac{r_1}{4}\bigr] \cup [4r_2,+\infty),
\end{cases}
\quad
\chi_2(s)=
\begin{cases}
\, 1, &\text{if } s\in\bigl[\frac{\rho_1}{2},2\rho_2\bigr],
\\
\, 0, &\text{if } s\in \bigl[0,\frac{\rho_1}{4}\bigr] \cup [4\rho_2,+\infty).
\end{cases}
\end{equation*}
Finally, for every $(t,x,p)\in\mathbb{R}\times\mathbb{R}^3\times\mathbb{R}^3$, we define
\begin{equation*}
\widehat{\mathcal{H}}_{\varepsilon}(t,x,p) = \widehat{G}(|p|)-\widehat{V}(|x|)+\widehat{\mathcal{R}}_{\varepsilon}(t,x,p),
\end{equation*}
where
\begin{align*}
&\widehat{G}(|p|) = \chi_2(|p|) G(|p|),
\\
&\widehat{V}(|x|) = \chi_1(|x|) V(|x|),
\\
&\widehat{\mathcal{R}}_{\varepsilon}(t,x,p) = \chi_1(|x|)\chi_2(|p|) \mathcal{R}_{\varepsilon}(t,x,p).
\end{align*}
Note that $\widehat{\mathcal{H}}_{\varepsilon}$ is two times differentiable with respect to $(x,p)$ and that both $\widehat{\mathcal{H}}_{\varepsilon}$ and their derivatives up to order two are continuous with respect to $(t,x,p)$ and globally bounded. Moreover, it easy to check that for every $\delta>0$ there exists $\tilde{\varepsilon}(\delta)>0$ such that for every $\varepsilon\in(-\tilde{\varepsilon}(\delta),\tilde{\varepsilon}(\delta))$ it holds that
\begin{equation}\label{eq-delta}
|\widehat{\mathcal{R}}_{\varepsilon}(t,x,p)| + |\mathrm{D}_{(x,p)}\widehat{\mathcal{R}}_{\varepsilon}(t,x,p)| \leq \delta,
\quad \text{for every $(x,p)\in\mathbb{R}^3\times\mathbb{R}^3$.}
\end{equation}

The next key result ensures that, under appropriate conditions, periodic solutions of system \eqref{eq-system}
can be obtained via periodic solutions of the modified Hamiltonian system
\begin{equation}\label{eq-system-mod}
\dot{x} = \nabla_{p} \widehat{\mathcal{H}}_{\varepsilon}(t,x,p),
\qquad
\dot{p} = -\nabla_{x} \widehat{\mathcal{H}}_{\varepsilon}(t,x,p).
\end{equation}
We stress that in the statement below $T$ stands for the (fixed) period of the Hamiltonian $\mathcal{H}_{\varepsilon}$ when we deal with the fixed-period problem; on the other hand $T$ is an unprescribed period when the Hamiltonian $\mathcal{H}_{\varepsilon}$ is autonomous and we deal with the fixed-energy problem.
Accordingly, if $z=(x,p) \in \mathcal{M}$, then we denote by $T(z)=T(x)$ the chosen period of $z$ in the fixed-energy problem (cf.~Theorem~\ref{th-abstract2}), while we agree that $T(z)=T$ in the fixed-period case (and, thus, the term $|T-T(z_*)|$ appearing in \eqref{eq-estimate-lemma} below is equal to zero).

\begin{lemma}\label{lemm-est}
For every $\sigma>0$ there exist $\hat{\varepsilon},\hat{\eta}>0$ such that, if $\varepsilon\in(-\hat{\varepsilon},\hat{\varepsilon})$, $\eta\in(0,\hat{\eta})$ and $z=(x,p)$ is a $T$-periodic solution of \eqref{eq-system-mod} satisfying 
\begin{equation}\label{eq-estimate-lemma}
\|z(Ts)-z_*(T(z_*)s)\|_{H^{1/2}_1}+ |T-T(z_*)|<\eta,
\end{equation}
for some $z_*=(x_*,p_*)\in\mathcal{M}$, then 
\begin{equation}\label{eq-thesis-lemma}
|x(t)-x_*(t)| + |p(t)-p_*(t)| < \sigma,
\end{equation}
for every $t\in[0,T]$. In particular, if $\sigma < \min\{r_1,\rho_1\}/2$, then $z = (x,p)$ is a $T$-periodic solution of \eqref{eq-system} (and, thus, $x$ is a $T$-periodic solution of \eqref{eq-main-bis}).
\end{lemma}

\begin{proof}
Given $\sigma>0$, let $z=(x,p)$ be a $T$-periodic solution of \eqref{eq-system-mod} satisfying \eqref{eq-estimate-lemma} for some $\eta>0$. 
We first impose that $\hat{\eta} < \min \{T(z) \colon z\in \mathcal{M}\} /2$ in such a way that
\begin{equation}\label{est-per}
\dfrac{1}{2} \min_{z\in \mathcal{M}} T(z) \leq T \leq \max_{z\in \mathcal{M}} T(z) + \dfrac{1}{2} \min_{z\in \mathcal{M}} T(z).
\end{equation}
By the continuous embedding of $H^1$ into $L^\infty$ we have
\begin{equation*}
|x(t)-x_*(t)| \leq C \bigl( \|x-x_*\|_{L^2(0,T)} + \|\dot{x}-\dot{x}_*\|_{L^2(0,T)} \bigr), \quad \text{for every $t \in \mathbb{R}$,}
\end{equation*}
for a suitable constant $C >0$ which can be taken independent of $T$ in view of \eqref{est-per}.
Using the equations satisfied by $x$ and $x_*$ (note that $G(|p_*|) = \widehat{G}(|p_*|)$) we find
\begin{align*}
\|\dot{x}-\dot{x}_*\|_{L^{2}(0,T)}
&\leq
\|\nabla\widehat{G}(|p|)-\nabla\widehat{G}(|p_*|)\|_{L^{2}(0,T)} + \|\nabla_{p}\widehat{\mathcal{R}}_\varepsilon(t,x,p)\|_{L^{2}(0,T)}
\\
&\leq M_1 \|p-p_*\|_{L^2(0,T)} + \delta T^{1/2},
\end{align*}
where in the last inequality we have used the global Lipschitz-continuity of $\nabla\widehat{G}$
and \eqref{eq-delta} assuming that $\varepsilon\in(-\tilde{\varepsilon}(\delta),\tilde{\varepsilon}(\delta))$.
Summing up, 
\begin{align*}
|x(t)-x_*(t)| 
&\leq C \|x-x_*\|_{L^2(0,T)} +
C M_1 \|p-p_*\|_{L^2(0,T)} + C \delta T^{1/2}
\\
&\leq C (1+M_1) \|z-z_*\|_{L^2(0,T)} + C \delta T^{1/2}, \quad \text{for every $t \in \mathbb{R}$.}
\end{align*}
Now we estimate
\begin{align*}
\|z-z_*\|_{L^2(0,T)}^2 
&= \int_0^T |z(t)-z_*(t)|^2 \,\mathrm{d}t
= T \int_0^1 |z(Ts)-z_*(Ts)|^2 \,\mathrm{d}s
\\
&\leq 2T \left( \int_0^1 |z(Ts)-z_*(T(z_*)s)|^2 \,\mathrm{d}s  + \int_0^1 |z_*(T(z_*)s)-z_*(Ts)|^2 \,\mathrm{d}s \right)
\\
&\leq 2T \left(\|z(Ts)-z_*(T(z_*)s)\|_{L^2(0,1)}^2 + \left(\max_{t\in\mathbb{R}} |\dot{z}_*(t)| \right)^{\!2} |T(z_*)-T|^2\right)
\\
&\leq 2T \left(\|z(Ts)-z_*(T(z_*)s)\|_{H^{1/2}_1}^2 + \left(\max_{z\in\mathcal{M}}\max_{t\in\mathbb{R}} |\dot{z}(t)|\right)^{\!2} |T(z_*)-T|^2\right)
\\
&\leq 2T \eta^2\left(1 + \left(\max_{z\in\mathcal{M}}\max_{t\in\mathbb{R}} |\dot{z}(t)|\right)^{\!2} \right).
\end{align*}
Recalling also \eqref{est-per}, we finally find
\begin{equation*}
|x(t)-x_*(t)| \leq C_1 \eta + C_2 \delta, \quad \text{for every $t\in[0,T]$,}
\end{equation*}
where $C_1,C_2$ are positive constants which depend only on $\mathcal{M}$.
A completely analogous argument leads to
\begin{equation*}
|p(t)-p_*(t)| \leq C_3 \eta + C_4 \delta, \quad \text{for every $t\in[0,T]$,}
\end{equation*}
where $C_3,C_4$ are positive constants which depend only on $\mathcal{M}$.
For small enough $\eta$ and $\delta$ (and thus $\varepsilon$) the estimate \eqref{eq-thesis-lemma} thus follows.
In particular, if $\sigma < \min\{r_1,\rho_1\}/2$, then, taking into account \eqref{def-r-rho}, we deduce that
\begin{equation*}
\dfrac{r_{1}}{2} < r_{1}-\sigma < |x(t)| < r_{2} + \sigma < 2r_2,
\qquad
\dfrac{\rho_{1}}{2} < \rho_{1}-\sigma < |p(t)| < \rho_{2} + \sigma < 2\rho_2,
\end{equation*}
for every $t\in [0,T]$ and thus $z=(x,p)$ solves the original system \eqref{eq-system}.
\end{proof}

\begin{proof}[Proof of Theorem~\ref{th-abstract1}]
We consider the Hamiltonian system \eqref{eq-system-mod} associated with $\widehat{\mathcal{H}}_{\varepsilon}$.
In view of the regularity and growth properties of $\widehat{\mathcal{H}}_{\varepsilon}$, 
as discussed in Section~\ref{section-3.1}, the $T$-periodic problem associated with \eqref{eq-system-mod} is equivalent to the search of critical points for the Hamiltonian action functional
$\mathcal{A}_{\varepsilon}\colon H^{1/2}_T \to \mathbb{R}$ given by
\begin{equation*}
\mathcal{A}_{\varepsilon}(z) = \frac{1}{2}  \ell_T(z,z) - \int_0^T \widehat{\mathcal{H}}_{\varepsilon}(t,z(t))\,\mathrm{d}t.
\end{equation*}
Choosing $\mathfrak{N} = \mathcal{M}$, we claim that Theorem~\ref{thastratto1} can be applied. 

Indeed, assumption $(i)$ trivially holds true since $\widehat{\mathcal{H}}_{0}$ and $\mathcal{H}_{0}$ coincide on an open region containing the orbit of any element of $\mathcal{M}$. 
For the same reason, condition $(iii)$ is equivalent to the non-degeneracy of $\mathcal{M}$ in view of Proposition~\ref{prop-nondegper}.
Finally, the validity of the Fredholm property in $(ii)$ has already been discussed in Section~\ref{section-3.1}. 

Let us fix $\sigma\in (0, \min\{r_1,\rho_1\}/2)$ and let $\hat{\varepsilon}$ and $\hat{\eta}$ be given by Lemma~\ref{lemm-est}.
Choosing $\eta\in(0,\hat{\eta})$ we consider the open neighborhood 
\begin{equation*}
U_{\eta} = \bigl\{ z\in H^{1/2}_T \colon \operatorname{dist}(z,\mathfrak{N}) < \eta \bigr\} \supset \mathfrak{N}.
\end{equation*}
Then Theorem~\ref{thastratto1} yields $\varepsilon^*\in(0,\hat{\varepsilon})$ such that, if $\varepsilon\in(-\varepsilon^*,\varepsilon^*)$, the Hamiltonian system \eqref{eq-system-mod} possesses at least $\operatorname{cat}(\mathfrak{N})$ $T$-periodic solutions belonging to $U_{\eta}$.
By Lemma~\ref{lemm-est}, these solutions solve system \eqref{eq-system} and satisfies \eqref{cond-th-abstract1}. The theorem is thus proved.
\end{proof}

\begin{proof}[Proof of Theorem~\ref{th-abstract2}]
We consider the Hamiltonian system \eqref{eq-system-mod} associated with $\widehat{\mathcal{H}}_{\varepsilon}$, which now does not depend on time.
In view of the regularity and growth properties of $\widehat{\mathcal{H}}_{\varepsilon}$, 
as discussed in Section~\ref{section-3.1} the fixed-energy (that is, $\widehat{\mathcal{H}}_{\varepsilon}(z)=h$) periodic problem associated with \eqref{eq-system-mod} is equivalent to the search of critical points for the free-period action functional
$\mathcal{B}_{\varepsilon}\colon H^{1/2}_1 \times (0,+\infty) \to \mathbb{R}$ given by
\begin{equation*}
\mathcal{B}_{\varepsilon}(\zeta,T) = \frac{1}{2}  \ell_1(\zeta,\zeta) - T \int_0^1 (\widehat{\mathcal{H}}_{\varepsilon}(\zeta(s))-h)\,\mathrm{d}s.
\end{equation*}
Denoting by $T(z)$ the chosen period $T(x)$ of $z=(x,p)\in\mathcal{M}$, we define the manifold
\begin{equation*}
\mathfrak{N} = \bigl\{(z(T(z)s),T(z)) \colon z\in\mathcal{M} \bigr\}.
\end{equation*}
Similarly as in the proof of Theorem~\ref{th-abstract1}, one can check that Theorem~\ref{thastratto1} applies; in particular the non-degeneracy assumption $(iii)$ holds by Proposition~\ref{prop-nondegene}.

Let us fix $\sigma\in (0, \min\{r_1,\rho_1\}/2)$ and let $\hat{\varepsilon}$ and $\hat{\eta}$ be given by Lemma~\ref{lemm-est}.
Choosing $\eta\in(0,\hat{\eta})$ we consider the open neighborhood 
\begin{equation*}
U_{\eta} = \bigl\{ (\zeta,T) \in H^{1/2}_1 \times (0,+\infty) \colon \operatorname{dist}((\zeta,T) ,\mathfrak{N}) < \eta \bigr\} \supset \mathfrak{N}.
\end{equation*}
Then Theorem~\ref{thastratto1} yields $\varepsilon^*\in(0,\hat{\varepsilon})$ such that, if $\varepsilon\in (-\varepsilon^*,\varepsilon^*)$, the 
functional $\mathcal{B}_{\varepsilon}$ possesses at least $1$ ($\leq \operatorname{cat}(\mathfrak{N})$, cf.~Remark~\ref{rem-cat-multiplicity}) critical point $(\zeta,T)\in U_{\eta}$. Then, the rescaled function $z(t) = \zeta(t/T)$ is a $T$-periodic solution of energy $h$ of the Hamiltonian system \eqref{eq-system-mod} satisfying \eqref{eq-estimate-lemma}.
By Lemma~\ref{lemm-est}, $z$ solves \eqref{eq-system} and satisfies \eqref{cond-th-abstract2} and $|T-T(x^*)|<\sigma$ provided that $\eta<\sigma$. The theorem is thus proved.
\end{proof}

\begin{remark}\label{rem-weinstein}
Since Theorem~\ref{th-abstract1} and Theorem~\ref{th-abstract2} are valid in any dimension (see Remark~\ref{rem-d-dim}) and for every non-degenerate periodic manifold $\mathcal{M}$, they can be applied in a variety of different situations, thus recovering some results already proved with other techniques.
For instance, they can be used to establish bifurcation from non-circular solutions in the plane, previously investigated in \cite{BDF-2} and \cite{BDF-3}, for the fixed-period and fixed-energy problems respectively, with arguments of Hamiltonian perturbation theory
(note that the non-degeneracy conditions on the Hamiltonian $\mathcal{K}_0$ in action-angle coordinates used in \cite{BDF-2,BDF-3} are equivalent to the non-degeneracy conditions of Theorem~\ref{th-abstract1} and Theorem~\ref{th-abstract2}). 
Moreover, Theorem~\ref{th-abstract2} can be applied to the case of bifurcation from circular solutions (in the plane or in the space) for the fixed-energy problem, recently tackled in \cite{BoFePa-25} using the Hamiltonian bifurcation theory developed in \cite{We-73,We-78}. Theorem~\ref{th-abstract1} and Theorem~\ref{th-abstract2} thus show that all these situations (circular/non-circular solutions, dimension two/three, fixed-period/fixed-energy) can be treated in a unified manner, always relying on the abstract bifurcation Theorem~\ref{th-abstract1}, via the $H^{1/2}$-Hamiltonian variational formulation described in this section. 

It is worth observing that Theorem~\ref{th-abstract2} could also be established as a direct corollary of the results in \cite{We-73,We-78}, providing bifurcation of fixed-energy solutions in the even more general setting of Hamiltonian systems on symplectic manifolds. The proof given in \cite{We-78} relies on variational arguments, as well; however, due to the more general and geometrical setting, $\mathcal{C}^\infty$-smoothness is assumed and a different functional setting is introduced: consequently a direct application of Theorem~\ref{th-abstract1} is not possible and the arguments become rather complicated. For these reasons, we thought it appropriate to provide, in the setting of equation \eqref{eq-main-bis}, an alternative and simpler treatment of the fixed-energy problem. We stress that this is done by using the $H^{1/2}$-Hamiltonian formulation, which is nowadays classical for the fixed-period problem: in this way, the asymmetry of the existing literature seems to be overcome.
\hfill$\lhd$
\end{remark}

\subsection{The main results}\label{section-4.2}

We now consider the unperturbed problem \eqref{eq-unperturbed} and assume the existence of a non-circular (thus non-constant) non-rectilinear periodic solution $x^*$ which, without loss of generality, is required to lie on the plane $x_3 = 0$.
Accordingly, we consider the manifold
\begin{equation}\label{def-M2-bis}
\mathcal{M}_2 = \bigl{\{} e^{i\phi} \pi(x^*(t-\theta)) \colon \phi,\theta \in \mathbb{R} \bigr{\}},
\end{equation}
where $\pi\colon \mathbb{R}^3 \to \mathbb{R}^2$ is defined as $\pi(x_1,x_2,x_3) = (x_1,x_2)$. Note that 
$\mathcal{M}_2$ is a manifold of non-circular periodic solutions of the unperturbed problem \eqref{eq-unperturbed} in dimension $2$, cf.~\eqref{def-M2}.

With this in mind, the following theorems can be stated. 

\begin{theorem}\label{th-main1}
Let $x^*$ be a non-circular non-rectilinear $T$-periodic solution for the unperturbed problem \eqref{eq-unperturbed} and assume that the manifold 
$\mathcal{M}_2$ defined in \eqref{def-M2-bis} is non-degenerate in the sense of Proposition~\ref{prop-nondegper}.
Then, for every $\sigma > 0$, there exists $\varepsilon^* > 0$ such that for every $\varepsilon \in (-\varepsilon^*,\varepsilon^*)$ there are five solutions $x_i$ of \eqref{eq-main-periodo} and $M_i \in O(3)$, $\theta_i \in \mathbb{R}$ ($i=1,\ldots,5$) such that
\begin{equation}\label{cond-th-main1}
\vert x_i(t) - M_i x^*(t - \theta_i) \vert < \sigma, \quad  \text{for every $t \in [0,T]$,}
\end{equation}
for $i=1,\ldots,5$.
\end{theorem}

\begin{theorem}\label{th-main2}
Let $x^*$ be a non-circular non-rectilinear $T^*$-periodic solution for the unperturbed problem \eqref{eq-unperturbed} (for some $T^*>0$) and assume that the manifold $\mathcal{M}_2$ defined in \eqref{def-M2-bis} is non-degenerate in the sense of Proposition~\ref{prop-nondegene}, with $h = \mathcal{E}_0(x^*(t),\dot x^*(t))$.
Then, for every $\sigma > 0$, there exists $\varepsilon^* > 0$ such that for every $\varepsilon \in (-\varepsilon^*,\varepsilon^*)$ there are 
a solution $(x,T)$ of \eqref{eq-main-energia} and $M \in O(3)$ such that $\vert T - T^* \vert < \sigma$ and 
\begin{equation}\label{cond-th-main2}
\vert x(t) - M x^*(t) \vert < \sigma, \quad  \text{for every $t \in [0,T]$.}
\end{equation}
\end{theorem}

\begin{remark}\label{rem-O3-SO3}
In view of the description of the periodic manifold $\mathcal{M}_3$ given in the proof of Proposition~\ref{prop-2.1}, in \eqref{cond-th-main1} one can take $M_i \in SO(3)$ and $\theta_i \in [0,\tau^*)$, where $\tau^*$ is the minimal period of $|x^*|$.
Similarly, in \eqref{cond-th-main2} it is possible to choose $M \in SO(3)$; notice that here a time-translation is not needed (that is, we can assume $\theta = 0$), since any time-translation of a solution of \eqref{eq-main-energia} is still a solution of \eqref{eq-main-energia}.
\hfill$\lhd$
\end{remark}

The proofs of Theorem~\ref{th-main1} and Theorem~\ref{th-main2} will be obtained by showing that the required non-degeneracy of $\mathcal{M}_2$
implies that the manifold
\begin{equation*}
\mathcal{M}_3 = \bigl{\{} M x^*(t-\theta) \colon M\in O(3), \theta \in \mathbb{R} \bigr{\}}
\end{equation*}
of periodic solutions of \eqref{eq-unperturbed} is non-degenerate as well, in such a way that Theorem~\ref{th-abstract1} or, respectively, Theorem~\ref{th-abstract2} can be applied.
Note that $\operatorname{cat}(\mathcal{M}_3) = 5$ by Proposition~\ref{prop-2.1} and this explains the number of solutions in the statement of Theorem~\ref{th-main1}.

\begin{proof}[Proof of Theorem~\ref{th-main1}]
Let $x\in \mathcal{M}_3$. Since $\operatorname{dim} (\mathcal{M}_3) = 4$ by Proposition~\ref{prop-2.1}, we need to show that the linearization along $x$ of the Hamiltonian system 
\begin{equation}\label{eq-HS0-sect4}
\begin{cases}
\, \dot{x} = \varphi^{-1}(p),
\vspace{4pt}\\
\, \dot{p} = V'(|x|) \dfrac{x}{|x|},
\end{cases}
\qquad (x,p) \in \bigl(\mathbb{R}^{3} \setminus \{0\} \bigr) \times \mathbb{R}^3,
\end{equation}
has a $4$-dimensional kernel of $T$-periodic solutions.
Without loss of generality, due to the invariances of the equation, we can assume that $x=x^*$, whose orbits lie in the plane $x_3=0$.
As discussed in Section~\ref{section-2.3}, in suitable partial action-angle coordinates, \eqref{eq-HS0-sect4} writes as
\begin{equation}\label{action-angle-3-sect4}
\begin{cases}
\, \dot I_i = 0, \qquad \dot \phi_i = \partial_{I_i} \mathcal{K}_0(I_1,I_2), \qquad i=1,2, \vspace{2pt}\\
\, \dot \Xi = 0, \qquad \dot \xi = 0,
\end{cases}
\end{equation}
where $\mathcal{K}_0$ is the Hamiltonian of the planar problem in action-angle coordinates and $I^*$ is the action value corresponding to $\mathcal{M}_2$, see Section~\ref{section-2.2}.
The monodromy $Q \colon \mathbb{R}^6\to\mathbb{R}^6$ associated to \eqref{action-angle-3-sect4} is 
\begin{equation*}
Q(X,Y,\alpha,\beta)
=(X, T \nabla^2\mathcal{K}_{0}(I^*)X+Y, \alpha, \beta),
\quad (X,Y,\alpha,\beta) \in \mathbb{R}^2\times\mathbb{R}^2\times\mathbb{R}\times\mathbb{R},
\end{equation*}
and thus
\begin{equation*}
\ker(I-Q) = \bigl{\{} (X,Y,\alpha,\beta) \in\mathbb{R}^6 \colon \nabla^2\mathcal{K}_{0}(I^*) X = 0 \bigr{\}}.
\end{equation*}
Since the periodic manifold $\mathcal{M}_2$ is non-degenerate by hypothesis, Proposition~\ref{prop-nondegper} $(iv)$ ensures that $\det \nabla^2 \mathcal{K}_0(I^*) \neq 0$. Therefore, 
\begin{equation*}
\ker(I-Q) = \{0\} \times \mathbb{R}^2 \times \mathbb{R} \times \mathbb{R}
\end{equation*}
and thus its dimension is four. 
Since this property is invariant by change of coordinates (cf.~\cite[Appendix~A]{BoFePa-25}), the proof is concluded.
\end{proof}

\begin{proof}[Proof of Theorem~\ref{th-main2}]
Similarly as in the proof of Theorem~\ref{th-main1}, we fix $x=x^*\in\mathcal{M}_3$ and, after passing to partial action-angle coordinates, we consider the monodromy
\begin{equation*}
Q(X,Y,\alpha,\beta)
=(X, T^* \nabla^2\mathcal{K}_{0}(I^*)X+Y, \alpha, \beta),
\quad (X,Y,\alpha,\beta) \in \mathbb{R}^2\times\mathbb{R}^2\times\mathbb{R}\times\mathbb{R}.
\end{equation*}
According to \cite[Appendix~A]{BoFePa-25}, the non-degeneracy condition can be verified directly for the operator $Q$. Precisely, setting
\begin{align*}
\mathcal{G} 
&= \bigl{\{} (X,Y,\alpha,\beta) \in T_{(I^*,\phi,\Xi,\xi)}\mathcal{K}_{0}^{-1}(h) \colon
\\ 
&\qquad (X,Y,\alpha,\beta) = Q(X,Y,\alpha,\beta) + \lambda (0,\nabla \mathcal{K}_0(I^*),0,0), \text{ for some $\lambda \in \mathbb{R}$} 
\bigr{\}}
\\
&=\bigl{\{} (X,Y,\alpha,\beta) \in \mathbb{R}^6 \colon
\langle (X, \nabla \mathcal{K}_{0}(I^*) \rangle = 0,
\\
&\qquad (0, T^* \nabla^2\mathcal{K}_{0}(I^*)X, 0,0)= \lambda (0,\nabla \mathcal{K}_0(I^*),0,0), \text{ for some $\lambda \in \mathbb{R}$}
\bigr{\}},
\end{align*}
the condition to be checked is now that $\operatorname{dim}(\mathcal{G})=\operatorname{dim}(\mathcal{M}_3)=4$.
Arguing exactly as in the final part of the proof of Proposition~\ref{prop-nondegene}, one can check that this is true if and only if \begin{equation*}
\operatorname{det}
\begin{pmatrix}
\nabla^2 \mathcal{K}_{0}(I^{*}) &\nabla \mathcal{K}_{0}(I^{*})^{\top} \vspace{3pt} \\
\nabla \mathcal{K}_{0}(I^{*}) & 0
\end{pmatrix} 
\neq 0,
\end{equation*}
that is, by Proposition~\ref{prop-nondegene}, if and only if $\mathcal{M}_2$ is non-degenerate.
The proof is thus complete.
\end{proof}

\section*{Acknowledgements}
We are grateful to Prof.~Alexander Dranishnikov for a useful discussion about the Lusternik--Schnirelmann category.





\bibliographystyle{elsart-num-sort}
\bibliography{BoFePa-biblio}

\begin{thebibliography}{10}
\expandafter\ifx\csname url\endcsname\relax
  \def\url#1{\texttt{#1}}\fi
\expandafter\ifx\csname urlprefix\endcsname\relax\def\urlprefix{URL }\fi

\bibitem{Ab-01}
A.~Abbondandolo, Morse theory for {H}amiltonian systems, vol. 425 of Chapman \&
  Hall/CRC Research Notes in Mathematics, Chapman \& Hall/CRC, Boca Raton, FL,
  2001.

\bibitem{AbMaPa-15}
A.~Abbondandolo, L.~Macarini, G.~P. Paternain, On the existence of three closed
  magnetic geodesics for subcritical energies, Comment. Math. Helv. 90 (2015)
  155--193.

\bibitem{AbSc-01}
A.~Abbondandolo, M.~Schwarz, A smooth pseudo-gradient for the {L}agrangian
  action functional, Adv. Nonlinear Stud. 9 (2009) 597--623.

\bibitem{AlFr-12}
P.~Albers, U.~Frauenfelder, Rabinowitz {F}loer homology: a survey, in: Global
  differential geometry, vol.~17 of Springer Proc. Math., Springer, Heidelberg,
  2012, pp. 437--461.

\bibitem{AmBe-92}
A.~Ambrosetti, U.~Bessi, Multiple closed orbits for perturbed {K}eplerian
  problems, J. Differential Equations 96 (1992) 283--294.

\bibitem{AmCo-89}
A.~Ambrosetti, V.~Coti~Zelati, Perturbation of {H}amiltonian systems with
  {K}eplerian potentials, Math. Z. 201 (1989) 227--242.

\bibitem{AmCo-93}
A.~Ambrosetti, V.~Coti~Zelati, Periodic solutions of singular {L}agrangian
  systems, vol.~10 of Progr. Nonlinear Differential Equations Appl.,
  Birkh\"{a}user, Boston, MA, 1993.

\bibitem{AmCoEk-87}
A.~Ambrosetti, V.~Coti~Zelati, I.~Ekeland, Symmetry breaking in {H}amiltonian
  systems, J. Differential Equations 67 (1987) 165--184.

\bibitem{Ar-89}
V.~I. Arnol'd, Mathematical methods of classical mechanics, vol.~60 of Graduate
  Texts in Mathematics, 2nd ed., Springer-Verlag, New York, 1989.

\bibitem{BaFuSa-18}
D.~Bambusi, A.~Fus\`e, M.~Sansottera, Exponential stability in the perturbed
  central force problem, Regul. Chaotic Dyn. 23 (2018) 821--841.

\bibitem{BA-23}
C.~Barrera-Anzaldo, Uniform bifurcation of comet-type periodic orbits in the
  restricted ({$n$} + 1)-body problem with non-{N}ewtonian homogeneous
  potential, J. Differential Equations 360 (2023) 572--598.

\bibitem{BaSz-05}
T.~Bartsch, A.~Szulkin, Hamiltonian systems: periodic and homoclinic solutions
  by variational methods, in: Handbook of differential equations: ordinary
  differential equations. {V}ol. {II}, Elsevier B. V., Amsterdam, 2005, pp.
  77--146.

\bibitem{Be-notes}
G.~Benettin, {A}ppunti per il corso di {M}eccanica {A}nalitica, {L}ecture notes
  (2017-2018).
\newline\urlprefix\url{https://www.math.unipd.it/~benettin/links-MA/ma-17_10_25.pdf}

\bibitem{BeFa-notes}
G.~Benettin, F.~Fass\`{o}, {I}ntroduzione ai sistemi dinamici, {L}ecture notes
  (2001-2002).
\newline\urlprefix\url{https://www.math.unipd.it/~benettin/postscript-pdf/hamilt.pdf}

\bibitem{BoDaFe-2021}
A.~Boscaggin, W.~Dambrosio, G.~Feltrin, Periodic solutions to a perturbed
  relativistic {K}epler problem, SIAM J. Math. Anal. 53 (2021) 5813--5834.

\bibitem{BDF-2}
A.~Boscaggin, W.~Dambrosio, G.~Feltrin, Periodic perturbations of central force
  problems and an application to a restricted 3-body problem, J. Math. Pures
  Appl. 186 (2024) 31--73.

\bibitem{BoDaFe-24ch}
A.~Boscaggin, W.~Dambrosio, G.~Feltrin, Prescribed energy periodic solutions of
  {K}epler problems with relativistic corrections, in: Topological methods for
  delay and ordinary differential equations with applications to continuum
  mechanics, vol.~51 of Adv. Mech. Math., Birkh\"auser/Springer, Cham, 2024,
  pp. 21--41.

\bibitem{BDF-3}
A.~Boscaggin, W.~Dambrosio, G.~Feltrin, Bifurcation of closed orbits of
  {H}amiltonian systems with application to geodesics of the {S}chwarzschild
  metric, SIAM J. Math. Anal.{}, to appear, arXiv:2310.02615.

\bibitem{BoDaMH-23}
A.~Boscaggin, W.~Dambrosio, E.~Mu\~noz Hern\'andez, A {M}aupertuis-type
  principle in relativistic mechanics and applications, Calc. Var. Partial
  Differential Equations 62 (2023) Paper No. 95, 29 pp.

\bibitem{BoFePa-25}
A.~Boscaggin, G.~Feltrin, D.~Papini, Nearly-circular periodic solutions of
  perturbed relativistic {K}epler problems: the fixed-period and the
  fixed-energy problems, Calc. Var. Partial Differential Equations 64 (2025)
  Paper No. 64, 29 pp.

\bibitem{BrHu-91}
H.~W. Broer, G.~B. Huitema, A proof of the isoenergetic {KAM}-theorem from the
  ``ordinary'' one, J. Differential Equations 90 (1991) 52--60.

\bibitem{Fa-05}
F.~Fass\`o, Superintegrable {H}amiltonian systems: geometry and perturbations,
  Acta Appl. Math. 87 (2005) 93--121.

\bibitem{FoGa-18}
A.~Fonda, A.~C. Gallo, Periodic perturbations with rotational symmetry of
  planar systems driven by a central force, J. Differential Equations 264
  (2018) 7055--7068.

\bibitem{FoGaGi-16}
A.~Fonda, M.~Garrione, P.~Gidoni, Periodic perturbations of {H}amiltonian
  systems, Adv. Nonlinear Anal. 5 (2016) 367--382.

\bibitem{FoTo-08}
A.~Fonda, R.~Toader, Periodic orbits of radially symmetric {K}eplerian-like
  systems: a topological degree approach, J. Differential Equations 244 (2008)
  3235--3264.

\bibitem{FoUr-17}
A.~Fonda, A.~J. Ure\~na, A higher dimensional {P}oincar\'e-{B}irkhoff theorem
  for {H}amiltonian flows, Ann. Inst. H. Poincar\'e{} C Anal. Non Lin\'eaire 34
  (2017) 679--698.

\bibitem{Ga-19}
A.~C. Gallo, Periodic solutions of perturbed central {H}amiltonian systems,
  NoDEA Nonlinear Differential Equations Appl. 26 (2019) Paper No. 34, 24 pp.

\bibitem{Ka-95}
T.~Kato, Perturbation theory for linear operators, Classics in Mathematics,
  Springer-Verlag, Berlin, 1995.

\bibitem{LC-28}
T.~Levi-Civita, Fondamenti di Meccanica Relativistica, Zanichelli, Bologna,
  1928.

\bibitem{MaWi-89}
J.~Mawhin, M.~Willem, Critical point theory and {H}amiltonian systems, vol.~74
  of Applied Mathematical Sciences, Springer-Verlag, New York, 1989.

\bibitem{MiFo-78}
A.~S. Mishchenko, A.~T. Fomenko, A generalized {L}iouville method for the
  integration of {H}amiltonian systems, Funkcional. Anal. i Prilo\v zen. 12
  (1978) 46--56, 96.

\bibitem{Mo-83}
L.~Montejano, A quick proof of {S}inghof's {${\rm cat}(M\times S\sp{1}) = {\rm
  cat}(M)+1$} theorem, Manuscripta Math. 42 (1983) 49--52.

\bibitem{MoZe-05}
J.~Moser, E.~J. Zehnder, Notes on dynamical systems, vol.~12 of Courant Lecture
  Notes in Mathematics, New York University, Courant Institute of Mathematical
  Sciences, New York; American Mathematical Society, Providence, RI, 2005.

\bibitem{OrRo-25}
R.~Ortega, D.~Rojas, Relativistic effects in the dynamics of a particle in a
  {C}oulomb field, Phys. D 472 (2025) Paper No. 134534, 9 pp.

\bibitem{Ra-79}
P.~H. Rabinowitz, Periodic solutions of a {H}amiltonian system on a prescribed
  energy surface, J. Differential Equations 33 (1979) 336--352.

\bibitem{RuSc-03}
Y.~B. Rudyak, F.~Schlenk, Lusternik-{S}chnirelmann theory for fixed points of
  maps, Topol. Methods Nonlinear Anal. 21 (2003) 171--194.

\bibitem{Si-79}
W.~Singhof, Minimal coverings of manifolds with balls, Manuscripta Math. 29
  (1979) 385--415.

\bibitem{So-14}
N.~Soave, Symbolic dynamics: from the {$N$}-centre to the {$(N+1)$}-body
  problem, a preliminary study, NoDEA Nonlinear Differential Equations Appl. 21
  (2014) 371--413.

\bibitem{We-73}
A.~Weinstein, Normal modes for nonlinear {H}amiltonian systems, Invent. Math.
  20 (1973) 47--57.

\bibitem{We-78}
A.~Weinstein, Bifurcations and {H}amilton's principle, Math. Z. 159 (1978)
  235--248.

\end{thebibliography}

\end{document}